\documentclass[preprint,12pt,numbers]{elsarticle}
\usepackage{amsmath,amssymb,amsfonts,amsthm,mathrsfs}
\usepackage{enumitem}
\usepackage[english]{babel}
\usepackage[colorlinks,linkcolor=blue,citecolor=blue,urlcolor=blue]{hyperref}
\biboptions{sort&compress}
\theoremstyle{plain}
\newtheorem{theorem}{Theorem}[section]
\newtheorem{lemma}[theorem]{Lemma}
\newtheorem{proposition}[theorem]{Proposition}

\theoremstyle{definition}

\newtheorem{remark}[theorem]{Remark}
\newtheorem{assumption}[theorem]{Assumption}
\numberwithin{equation}{section}
\newcommand{\R}{\mathbb R}
\newcommand{\E}{\mathbb E}
\newcommand{\PP}{\mathbb P}
\newcommand{\dd}{\,\mathrm d}
\newcommand{\eps}{\varepsilon}
\allowdisplaybreaks[2]
\journal{Journal of Functional Analysis}
\hypersetup{pdftitle={Global well-posedness and scattering for energy-critical Schrodinger equations with conservative Marcus noise},pdfauthor={Bin Liu, Fan Xu, Bo Yang, Lei Zhang},pdfsubject={Research manuscript prepared for Journal of Functional Analysis}}

\begin{document}
\begin{frontmatter}
\title{Global well-posedness and scattering for energy-critical
Schr\"odinger equations with conservative Marcus noise}
\author[a,b]{Bin Liu}
\ead{binliu@mail.hust.edu.cn}
\author[a,b]{Fan Xu}
\ead{2026510746@hust.edu.cn}
\author[a,b]{Bo Yang}
\ead{boyang_0717@hust.edu.cn}
\author[a,b]{Lei Zhang}
\ead{lei_zhang@hust.edu.cn}
\affiliation[a]{organization={School of Mathematics and Statistics,
Huazhong University of Science and Technology},city={Wuhan},
postcode={430074},state={Hubei},country={P. R. China}}
\affiliation[b]{organization={Hubei Key Laboratory of Engineering
Modeling and Scientific Computing, Huazhong University of Science
and Technology},city={Wuhan},postcode={430074},state={Hubei},
country={P. R. China}}

\begin{abstract}
This work establishes global well-posedness for the three-dimensional
defocusing energy-critical nonlinear Schr\"odinger equation with
conservative linear Marcus noise. For arbitrary deterministic
energy-space data, we construct a probabilistically strong solution,
unique in the critical spacetime class, and prove convergence of
finite-activity approximations. The spatial profiles need only be
bounded and Lipschitz, while the centered pure-jump L\'evy noise is
assumed to have finite second moment. If the temporal noise coefficient
is square integrable on the half-line, the solution scatters almost
surely, with uniform mean-square tail convergence and no smallness
restriction. The solution conserves mass and satisfies an exact
energy balance. The proof combines deterministic energy-critical
stability with Poisson Strichartz estimates; convergence of the
interaction-representation martingale gives the critical spacetime
control needed for continuation and scattering.
\end{abstract}

\end{frontmatter}
\clearpage

\section{Introduction and main results}

We study critical spacetime estimates for conservative random phase
perturbations of the three-dimensional defocusing energy-critical
nonlinear Schr\"odinger equation (NLS)
\begin{equation}\label{det:eq}
 i\partial_t v+\Delta v=|v|^4v.
\end{equation}
The scaling $v(t,x)\mapsto\rho^{1/2}v(\rho^2t,\rho x)$ preserves
the homogeneous energy norm. Global continuation at this regularity
requires control of a critical spacetime norm whose local smallness
does not follow from an energy bound alone. Conservative random
phase impulses introduce a further difficulty: they preserve mass,
but the energy itself evolves with the noise. The problem is therefore
to combine deterministic critical stability with estimates for both
the random energy increments and the stochastic spacetime norm under
a common set of assumptions on the driving process.

The deterministic foundation is the large-data scattering theorem
for \eqref{det:eq}. The local critical theory originates in
\cite{CazenaveWeissler1990}, and the linear estimates used here
follow from the endpoint Strichartz theorem of Keel and Tao
\cite{KeelTao1998}. Radial global results were obtained by
Bourgain \cite{Bourgain1999} and Grillakis \cite{Grillakis2000};
Tao \cite{TaoRadial2005} proved radial scattering with
quantitative spacetime bounds in higher dimensions. The
nonradial three-dimensional theorem of Colliander, Keel,
Staffilani, Takaoka and Tao \cite{CKSTT2008} supplies the global
comparison solution and its critical spacetime bound. Our argument
builds on this theorem by controlling the perturbation produced by
the random phase flow. The corresponding
four- and higher-dimensional results are
\cite{RyckmanVisan2007,Visan2007}. Stability at energy criticality
is studied in \cite{TaoVisan2005}; for the three-dimensional
quintic equation, the differentiated polynomial estimates allow
a direct proof in the spaces used below. We provide that proof
with its dependence on the comparison bound, using the local
theory described in \cite{Cazenave2003,Tao2006}. The defocusing
sign supplies comparison bounds for arbitrary energy. In the
focusing problem, the theory of \cite{KenigMerle2006} instead
involves a threshold determined by the ground state.

The random phase flow is naturally associated with a real
time-dependent potential, which changes the local phase of a wave.
For a smooth signal $\ell=(\ell_1,\ldots,\ell_m)$, the perturbed
equation takes the form
\[
 i\partial_t u+\Delta u=|u|^4u+
 a(t)\sum_{j=1}^m V_j(x)u\,\dot\ell_j(t),
\]
where the real functions $V_j$ describe the spatial variation of
the perturbation. An impulse with increment $z$ acts by the flow
\[
 w(x)\longmapsto\exp\bigl(-ia(t)z\cdot V(x)\bigr)w(x).
\]
This flow preserves the pointwise modulus and is the jump rule
associated with the Marcus canonical integral
\cite{Marcus1981,KurtzPardouxProtter1995}. The potential energy and
the mass are consequently unchanged at a jump. The kinetic energy,
however, changes through the spatial gradient of the phase.

Replacing the smooth signal by a centered pure-jump L\'evy process
leads to
\begin{equation}\label{main:canonical}
 \left\{\begin{aligned}
 \mathrm du&=i(\Delta u-|u|^4u)\dd t
       -ia(t)\sum_{j=1}^m V_j u(t-)\diamond\mathrm dL_j(t),\\
 u(0)&=u_0\in H^1(\R^3;\mathbb C).
 \end{aligned}\right.
\end{equation}
Here $t\ge0$, $x\in\R^3$, and $\diamond$ denotes the Marcus
interpretation specified by the phase flow. The process $L$ has
finite second moment, the spatial profiles belong to
$W^{1,\infty}(\R^3;\R)$, and $a$ is a deterministic locally
square-integrable function. The precise compensated equation is
given below. Centering fixes the deterministic drift of the driving
process, which must be specified for the long-time problem.

Our main result constructs global, pathwise unique solutions of
\eqref{main:canonical} for arbitrary initial data in $H^1(\R^3)$,
the finite-mass part of the homogeneous energy space. The finite-mass
assumption enters the estimates for the spatially varying phase
multipliers. The same estimates give convergence of finite-activity
approximations on the prescribed stochastic basis. The jump measure may be
nonsymmetric and have unbounded support or infinite first absolute
moment near zero. The spatial profiles require no decay at infinity.
When $a$ is square integrable on the whole half-line, we also obtain
almost sure scattering and convergence in the mean-square norm of
the entire late-time error, without a smallness restriction on the
initial data or the noise intensity.

The main issue in coupling critical stability to the random forcing
is to retain the second-moment assumption throughout the construction.
The critical spacetime exponent is larger than two, and the energy
increment contains a term quadratic in the jump size. We estimate the
Poisson convolution in the second moment in probability, with values
in a spacetime Banach space. In the energy balance, we separate a
martingale term linear in the mark from an increasing process
quadratic in the mark.
Only the first moment of the latter process is needed. These two
estimates provide uniform control as the small-jump cutoff is
removed. Deterministic critical stability then supplies the
continuation argument.

The strong theory for jump-driven NLS was developed in
\cite{BrzezniakLiuZhu2021}, where stochastic Strichartz estimates
for Poisson convolutions yield global solutions with nonlinear
conservative Marcus coefficients throughout the strictly
mass-subcritical range, for both signs of the power nonlinearity.
Additive pure-jump forcing in the same range is treated in
\cite{WangZhaiZhu2023}. Energy-space martingale solutions for
whole-space equations with jump noise were constructed in
\cite{deBouardHausenblas2019}; pathwise uniqueness and strong
solutions were studied in \cite{deBouardHausenblasOndrejat2019}.
For conservative linear Marcus noise, energy-subcritical martingale
solutions on compact manifolds and bounded domains are obtained
in \cite{BrzezniakHornungManna2020}. These constructions use
different regularity, moment and compactness assumptions. The
energy-critical whole-space problem additionally requires control
of a spacetime norm whose local smallness does not follow from
an energy bound alone.

We use the Poisson Strichartz method of
\cite{BrzezniakLiuZhu2021} together with the second-moment
inequality in the Banach-space integration theory of
\cite{DirksenMaasVanNeerven2013}. The distinction between
probability moments and jump integrability is developed further
in \cite{Dirksen2014}, while maximal estimates and c\`adl\`ag
versions for stochastic convolutions are studied in
\cite{ZhuBrzezniakHausenblas2017}. Our approximations retain the
original Poisson random measure and converge by stability estimates.
This directly preserves the noise coordinate, whose treatment in
Skorokhod representation arguments is discussed in
\cite{OndrejatSeidler2025}.

For Gaussian forcing, critical stochastic NLS already has a
global theory. Oh and Okamoto \cite{OhOkamoto2020} prove global
well-posedness for additive noise at mass criticality and at
energy criticality in dimensions three through six. Their proof
combines stochastic mass or energy bounds with deterministic
perturbation theory. Zhang \cite{Zhang2023} establishes the
corresponding mass- and energy-critical theory for linear
multiplicative Wiener noise, together with rescaled scattering
and a support theorem. The unconditional energy-critical global
result covers dimensions three through six; higher dimensions
involve an additional a priori energy-space bound. The spatial
profiles are smooth, with asymptotic flatness assumptions on
their positive-order derivatives. The present work treats the
discontinuous phase flow directly at one spatial derivative and
with only a second moment of the jump measure.

The Gaussian theories build on the mass- and energy-space results
of de Bouard and Debussche
\cite{deBouardDebussche1999,deBouardDebussche2003}. Stochastic
Strichartz estimates lead to the compact-manifold theory in
\cite{BrzezniakMillet2014} and the whole-space theory in
\cite{Hornung2018}. The latter includes local mass-critical and
global mass-subcritical solutions with nonlinear Stratonovich
noise. A different construction uses stochastic rescaling to
obtain random dispersive equations with lower-order coefficients
\cite{BarbuRocknerZhang2016}; pathwise Strichartz and local
smoothing estimates for the resulting operators are developed
in \cite{Zhang2022}. In the one-dimensional mass-critical case,
global well-posedness and stability at $L^2$ regularity were
proved in \cite{FanXu2021}, and approximation of the driving
signal was treated in \cite{FanXuWongZakai2021}. The quintic
nonlinearity is also present there, but the three-dimensional
problem requires estimates for one spatial derivative of both
the nonlinearity and the noise.

Other stochastic results concern distinct critical settings.
The almost-conservation method in \cite{CheungLiOh2021} gives
global solutions below the energy space for the three-dimensional
cubic equation with additive Gaussian forcing. The
four-dimensional energy-critical problem with a nonzero background
is treated in \cite{CheungLi2022}, and the additively forced
quintic equation on the three-dimensional torus in
\cite{LiOkamotoTao2025}. Local energy-critical theories and
focusing blow-up are studied in \cite{MilletRoudenko2026}.
Regularization by sufficiently strong nonconservative noise at
criticality is established in \cite{SpitzZhangZhao2026}.
Equation~\eqref{main:canonical} instead preserves mass, and its
energy has a nonnegative expected increment. Its global
construction uses the defocusing deterministic dynamics and
the control of these random energy increments.

For long-time behavior, Herr, R\"ockner and Zhang
\cite{HerrRocknerZhang2019} developed scattering for
multiplicative Wiener equations under finite global quadratic
variation, including an energy-critical statement conditional
on global existence and the required Strichartz regularity.
The critical theory in \cite{Zhang2023} supplies global solutions
in the dimensions specified above. Persistent small noise is
treated for the three-dimensional mass-critical equation in
\cite{FanXuZhao2023}. Our scattering theorem concerns finite
total temporal noise intensity. Under this condition, the
interaction-representation martingale converges, and the
restarted Poisson convolution becomes small on late tails in
the critical solution norm. This gives free scattering for
the original jump-driven solution. The maximal energy estimate
also yields mean-square tail convergence and identifies the
expected kinetic energy of the scattering state.

\subsection{The model and the main statements}

Fix a positive integer $m$. Let
$(\Omega,\mathcal F,(\mathcal F_t)_{t\ge0},\PP)$ be a complete probability
space with a complete, right-continuous filtration, and let $\E$
denote expectation with respect to $\PP$.
Let $N$ be an $(\mathcal F_t)$-Poisson random measure on
$[0,\infty)\times Z$, where $Z=\R^m\setminus\{0\}$, with compensator
$\dd t\,\nu(\dd z)$, and write $\widetilde N=N-\dd t\,\nu$.
All stochastic integrals below use this fixed stochastic basis.

\begin{assumption}\label{ass:model}
The measure $\nu$ is a L\'evy measure satisfying
\begin{equation}\label{noise:moment}
 Q:=\int_Z |z|^2\nu(\dd z)<\infty.
\end{equation}
The spatial coefficients $V_1,\ldots,V_m$ are real members of
$W^{1,\infty}(\R^3)$. The temporal coefficient $a:[0,\infty)\to\R$ is a
deterministic Borel function, finite everywhere, with
$a\in L^2(0,T)$ for every finite $T$.
\end{assumption}

We choose the centered square-integrable L\'evy martingale
\begin{equation}\label{noise:centered}
 L(t)=\int_0^t\int_Z z\,\widetilde N(\dd s,\dd z).
\end{equation}
The integral is defined in $L^2(\Omega;\R^m)$; no first-moment condition
near zero is needed. Changing $a$ on a deterministic null set does not
change any solution. Put $S(t)=e^{it\Delta}$ and, for $z\in Z$, define
\begin{equation}\label{noise:maps}
 \begin{split}
 A_z(x)&=\sum_{j=1}^m z_jV_j(x),\qquad
 J_{t,z}f=e^{-ia(t)A_z}f,\\
 G_{t,z}f&=(e^{-ia(t)A_z}-1)f,\qquad
 R_{t,z}f=(e^{-ia(t)A_z}-1+ia(t)A_z)f,\\
 B(t)f&=\int_Z R_{t,z}f\,\nu(\dd z).
 \end{split}
\end{equation}
The integral defining $B(t)$ is a Bochner integral in $H^1(\R^3)$.
Equation~\eqref{main:canonical} means the mild identity
\begin{equation}\label{main:mild}
 \begin{split}
 u(t)={}&S(t)u_0-i\int_0^t S(t-s)|u(s)|^4u(s)\dd s\\
 &+\int_0^t\int_Z S(t-s)G_{s,z}u(s-)\,\widetilde N(\dd s,\dd z)
       +\int_0^t S(t-s)B(s)u(s)\dd s.
 \end{split}
\end{equation}
The two correction terms in this formula are part of the model, and
both are retained in every approximation.

An adapted mild solution means an adapted $H^1(\R^3)$-valued process
with right-continuous paths and left limits, locally belonging to the
spacetime class in \eqref{sol:class}, for which \eqref{main:mild}
holds in $H^1(\R^3)$ simultaneously at all times before its lifetime,
outside one null set. Pathwise uniqueness means that two such solutions
with the same initial datum and the same Poisson random measure agree
up to their common lifetime. The term probabilistically strong refers
to construction on the prescribed stochastic basis; it does not assert
that the Laplacian is an $H^1$-valued process.

The Poisson integral in this definition is initially a localized
Hilbert-space stochastic integral. In fact, Lemma~\ref{lem:maps} gives,
on every compact interval $[0,T]$ on which the path exists,
\begin{equation}\label{mild:local-integrability}
 \begin{split}
 \int_0^T\int_Z\|G_{s,z}u(s-)\|_{H^1(\R^3)}^2\nu(\dd z)\dd s
 &\le C_V A(T)\sup_{0\le s\le T}\|u(s)\|_{H^1(\R^3)}^2,\\
 \int_0^T\|B(s)u(s)\|_{H^1(\R^3)}\dd s
 &\le C_V A(T)\sup_{0\le s\le T}\|u(s)\|_{H^1(\R^3)}.
 \end{split}
\end{equation}
Here $A(T)$ and the coefficient constant $C_V$ are specified below.
Both right-hand sides are finite pathwise. Stopping the accumulated
quadratic integral therefore defines the Poisson term without an
advance moment assumption on the solution. Its integrand uses the
predictable left-limit process $u(s-)$, with $u(0-)=u_0$.
The nonlinear time integral is defined by dual Strichartz estimates,
as detailed in Lemmas~\ref{lem:strichartz} and \ref{lem:quintic}.

Define the mass and defocusing energy by
\begin{equation}\label{mass:energy}
 M(f)=\|f\|_{L^2(\R^3)}^2,\qquad
 E(f)=\frac12\|\nabla f\|_{L^2(\R^3)}^2
      +\frac16\|f\|_{L^6(\R^3)}^6.
\end{equation}
We shall use the constant
$b_1=(\sum_{j=1}^m\|\nabla V_j\|_{L^\infty(\R^3)}^2)^{1/2}$
and the deterministic quantity
\begin{equation}\label{clock:def}
 A(T)=Q\int_0^T a(s)^2\dd s.
\end{equation}

\begin{theorem}\label{thm:global}
Under Assumption~\ref{ass:model}, for every deterministic
$u_0\in H^1(\R^3)$ there is a global adapted solution of
\eqref{main:mild}. For every finite $T$, almost surely,
\begin{equation}\label{sol:class}
 u\in D([0,T];H^1(\R^3))
       \cap L^{10}(0,T;W^{1,30/13}(\R^3)).
\end{equation}
It is pathwise unique among adapted mild solutions in this class.
Moreover, almost surely $M(u(t))=M(u_0)$ for all $t\ge0$, and
\begin{equation}\label{energy:bound:intro}
 \E\sup_{0\le t\le T}E(u(t))
       \le C\bigl(E(u_0)+b_1^2M(u_0)A(T)\bigr),
\end{equation}
where $C$ is an absolute constant.

Let $\eps_n\downarrow0$, and let $u_n$ solve \eqref{main:mild} with
all mark integrals restricted to $Z_n=\{|z|>\eps_n\}$. Then
\begin{equation}\label{approx:intro}
 \sup_{0\le t\le T}\|u_n(t)-u(t)\|_{H^1(\R^3)}
 +\|u_n-u\|_{L^{10}(0,T;W^{1,30/13}(\R^3))}
 \longrightarrow0
\end{equation}
in probability. The same convergence holds for solutions on the same
stochastic basis when deterministic initial data converge in $H^1(\R^3)$.
\end{theorem}

\begin{remark}
Theorem~\ref{thm:global} establishes a strong energy-space theory
at the critical power for spatially varying linear Marcus noise.
The mass-space result in \cite{BrzezniakLiuZhu2021} allows nonlinear
Marcus coefficients and treats strictly mass-subcritical powers;
the additive theory in \cite{WangZhaiZhu2023} has the same restriction
on the power. The present linear phase structure permits one
spatial derivative and an exact kinetic-energy calculation under
\eqref{noise:moment}. The resulting uniform energy estimate is
combined with \cite{CKSTT2008} to control the critical spacetime
norm. Compared with the energy-subcritical constructions in
\cite{deBouardHausenblas2019,BrzezniakHornungManna2020}, the
convergence argument here takes place on the original stochastic
basis and identifies the limit through pathwise uniqueness in
\eqref{sol:class}.

The Gaussian energy-critical theories
\cite{OhOkamoto2020,Zhang2023} already establish the corresponding
critical well-posedness question for additive and linear multiplicative
Wiener forcing under their respective assumptions. Our conclusion
concerns a different noise class: the Marcus phase map is treated
directly, requiring only bounded Lipschitz spatial profiles and a
second moment of the centered jump measure. No spatial decay, jump
symmetry, bounded jump support or fourth jump moment is imposed.
The contribution is the critical global construction under these
spatial and moment assumptions, together with stability of its
finite-activity approximations. The deterministic comparison theorem
acts in the full homogeneous energy space, whereas our phase-multiplier
and energy estimates also use the conserved mass. The stochastic
result therefore concerns the finite-mass energy space $H^1(\R^3)$.
\end{remark}

For later use, define the current
$j(f)=\operatorname{Im}(\overline f\,\nabla f)$ and the real quantities
\begin{equation}\label{energy:increments}
 \Lambda_{t,z}(f)=-a(t)\int_{\R^3}\nabla A_z\cdot j(f)\dd x,
 \qquad
 D_{t,z}(f)=\frac{a(t)^2}{2}
              \|f\nabla A_z\|_{L^2(\R^3)}^2.
\end{equation}

\begin{proposition}\label{prop:balance}
The solution of Theorem~\ref{thm:global} satisfies, indistinguishably,
\begin{equation}\label{energy:balance:intro}
 \begin{split}
 E(u(t))={}&E(u_0)
 +\int_0^t\int_Z\Lambda_{s,z}(u(s-))\,\widetilde N(\dd s,\dd z)
 \\
 &+\int_0^t\int_ZD_{s,z}(u(s-))\,N(\dd s,\dd z).
 \end{split}
\end{equation}
The first integral is a square-integrable real martingale on each
bounded interval. The second is nonnegative and integrable.
\end{proposition}

\begin{remark}
The decomposition in Proposition~\ref{prop:balance} separates the
two orders of the energy increment in the jump mark. The linear
part is a square-integrable martingale after the energy estimate,
and the quadratic part defines an integrable increasing process.
Its expectation uses precisely the second jump moment in
\eqref{noise:moment}. Treating the compensated quadratic part by
a square-integrable martingale estimate would instead introduce
a fourth jump moment. Keeping it against $N$ preserves the
second-moment hypothesis in both the convolution estimate and the
energy balance. This use of distinct moment bounds is consistent
with \cite{DirksenMaasVanNeerven2013,Dirksen2014}. The increasing
term also gives the expected energy supplied by the spatially
varying phase impulses.
\end{remark}

\begin{theorem}\label{thm:scattering}
Assume in addition that $a\in L^2(0,\infty)$. There is an
$H^1(\R^3)$-valued random variable $u_+$ such that, almost surely,
\begin{equation}\label{scatter:intro}
 \|u\|_{L^{10}(0,\infty;W^{1,30/13}(\R^3))}<\infty,
 \qquad
 \lim_{t\to\infty}\|u(t)-S(t)u_+\|_{H^1(\R^3)}=0.
\end{equation}
Furthermore, $u_+\in L^2(\Omega;H^1(\R^3))$, $M(u_+)=M(u_0)$ almost
surely, and
\begin{equation}\label{scatter:mean}
 \lim_{T\to\infty}\E\sup_{t\ge T}
           \|u(t)-S(t)u_+\|_{H^1(\R^3)}^2=0.
\end{equation}
The asymptotic kinetic energy obeys the identity
\begin{equation}\label{scatter:energy}
 \frac12\E\|\nabla u_+\|_{L^2(\R^3)}^2
 =E(u_0)+\frac12\E\int_0^\infty a(s)^2
       \int_Z\|u(s)\nabla A_z\|_{L^2(\R^3)}^2\nu(\dd z)\dd s.
\end{equation}
\end{theorem}

\begin{remark}
Theorem~\ref{thm:scattering} has no smallness assumption on
$\|a\|_{L^2(0,\infty)}$, on the initial data, or on the spatial profiles.
Finite total noise intensity is closely related to the long-time regime
in \cite{HerrRocknerZhang2019,Zhang2023}. Here the second-moment
condition gives convergence of the interaction-representation
martingale and makes the remaining Poisson convolution small in the
critical solution space. This yields scattering for the original
jump-driven solution with unbounded jump sizes and one bounded
spatial derivative. For conservative Gaussian models, a tail gauge
converging to the identity can relate the rescaled and original
scattering formulations; the distinction here lies in the noise and
coefficient assumptions.

The persistent-noise result in \cite{FanXuZhao2023} addresses a
different balance between dispersion and forcing: it uses a small
noise coefficient and the three-dimensional mass-critical power.
Our finite-intensity hypothesis permits an arbitrary amount of
forcing on an initial interval. The mean-square tail convergence
and the identity \eqref{scatter:energy} additionally identify the
energy carried by the limiting free state in terms of the energy
injected by the spatially varying jumps.
\end{remark}

\subsection{Proof strategy and notation}

The proof couples stochastic moment estimates with pathwise critical
comparison. The defocusing energy and conserved mass control the
$H^1$ norm. The deterministic large-data theorem bounds the critical
spacetime norm of a comparison solution. The Poisson estimates must
then make the difference between these two evolutions small on the intervals where
comparison is used. The first bound alone does not rule out divergence
of the critical spacetime norm at a finite lifetime.

For the local construction, we cut off the accumulated critical
spacetime norm. The cutoff must be included in the difference
estimate: two paths generally reach the same cutoff level at
different times. After this estimate is proved, the deterministic
Duhamel term has a small Lipschitz constant when the cutoff level is
small. The stochastic term is controlled on intervals with small
quadratic noise intensity. Linearity of the phase coefficient gives
the global $H^1$ Lipschitz bound for $G_{t,z}=J_{t,z}-I$, with size
proportional to $|a(t)||z|$. Its second-order remainder has size
$a(t)^2|z|^2$, so its compensator is integrable under
\eqref{noise:moment}. This separates the critical nonlinear
smallness from the smallness used for stochastic contraction.
It also specifies which estimates survive removal of the small
jumps.

The stochastic spacetime estimate is based on the Poisson
Strichartz method of \cite{BrzezniakLiuZhu2021}. One first regards
the retarded free evolution of an integrand as a random variable
with values in the spacetime Banach space. The deterministic
homogeneous estimate bounds its norm, and a martingale-type-two
inequality controls its second moment. This argument uses the
Banach-space Poisson integration framework in
\cite{DirksenMaasVanNeerven2013}; the $H^1$ path supremum is
controlled separately by the Hilbert-space martingale inequality.
The spatial and temporal exponents in the Strichartz norm therefore
do not impose corresponding higher moments on the jump measure.

The energy identity is justified first for finite-activity
approximations. Between jumps the compensator contributes a real
time-dependent potential, whose energy derivative must be retained.
At a jump the phase map preserves the potential energy and changes
the kinetic energy by an explicit linear term and an explicit
quadratic term. Combining the continuous potential contribution
with the compensator of the linear jump term gives the balance in
Proposition~\ref{prop:balance}. Keeping the quadratic term against
$N$ makes its expected total variation depend only on the second
jump moment. A martingale maximal inequality and absorption then
give a bound uniform in the finite-activity cutoff. Local convergence
first transfers this bound to the maximal solution. After global
existence has been established, the exact identity is passed to the
limit using common stopping times that bound the $H^1$ norms of the
approximations and the solution. The two limit passages are carried
out in that order.

To continue this solution, stop the integrand before a hypothetical
finite lifetime and conjugate the stochastic convolution by the
free group. The resulting $H^1$-valued martingale has a left limit.
The Poisson Strichartz estimate also places its free evolution in
the critical spacetime class. When the equation is restarted near
the lifetime, the contribution of the noise accumulated before
the restart is subtracted. The remainder is small both in the
$H^1$ path supremum and in the critical spacetime norm. The latter
assertion uses the terminal martingale value and continuity of the
deterministic homogeneous solution operator. The compensator term
is small by absolute continuity of its time integral.

We can now compare with the global deterministic solution having
the same datum at the restart. The uniform energy bound gives a
uniform bound for this deterministic comparison, by
\cite{CKSTT2008}. Critical stability then controls the stochastic
solution up to the alleged lifetime and allows continuation. This
is the perturbative globalization principle also present in
\cite{FanXu2021,OhOkamoto2020}; here the forcing is multiplicative,
has jumps, and is estimated after one spatial derivative. The
left-limit argument is what supplies a pathwise restart compatible
with those estimates.

For scattering, finite total noise intensity extends the maximal
energy estimate to the whole half-line. The conjugated martingale
then converges at infinity, and its restarted tail tends to zero
in the critical solution norm. The same deterministic comparison
applies from a sufficiently late, path-dependent time. This yields
a finite global critical norm and hence convergence of the
nonlinear Duhamel integral in $H^1$. The limiting stochastic,
nonlinear and compensator integrals define the scattering state.
Finally, the integrable maximal energy bound permits passage from
pathwise scattering to the mean-square tail estimate. The decay
of the $L^6$ norm of a free $H^1$ solution identifies the limiting
kinetic energy. The proofs of these statements are given in
Section~\ref{sec:scattering}.

Unless otherwise stated, spatial function spaces are complex-valued.
We take derivatives of real-valued functionals over the underlying real
space. In particular, the real $L^2$ pairing is
$\langle f,g\rangle_{\mathrm{re}}
=\operatorname{Re}\int_{\R^3}f\overline g\dd x$.
We use $H^1(\R^3)=W^{1,2}(\R^3)$ with
\[
 \|f\|_{H^1(\R^3)}^2
   =\|f\|_{L^2(\R^3)}^2+\|\nabla f\|_{L^2(\R^3)}^2.
\]
For $1\le r<\infty$, the $W^{1,r}(\R^3)$ norm is
$(\|f\|_{L^r(\R^3)}^r+
\sum_{k=1}^3\|\partial_k f\|_{L^r(\R^3)}^r)^{1/r}$;
equivalent finite-dimensional norms are used for vector-valued functions.
The notation $\dot H^1(\R^3)$ refers to the homogeneous Sobolev space
with norm $\|\nabla f\|_{L^2(\R^3)}$.
For an interval $I$, set
\begin{equation}\label{spaces:def}
 \begin{split}
 Y(I)&=L^{10}(I;W^{1,30/13}(\R^3)),\qquad
 \mathcal N(I)=L^2(I;W^{1,6/5}(\R^3)),\\
 \mathcal X(I)&=D(I;H^1(\R^3))
       \cap L^\infty(I;H^1(\R^3))\cap Y(I),\\
 \|f\|_{\mathcal X(I)}&=
       \sup_{t\in I}\|f(t)\|_{H^1(\R^3)}+\|f\|_{Y(I)}.
 \end{split}
\end{equation}
Here $D$ denotes right-continuous paths with left limits. On half-open
intervals the left limit at the right endpoint is not included in the
definition. The boundedness condition is relevant when $I$ is
half-open or unbounded; the displayed supremum uses the actual
c\`adl\`ag representative, including every endpoint belonging to $I$.
We use the same norm for continuous paths. An expression such as
$Y(s,b)$ or $\mathcal N(s,b)$ means the corresponding space on
$(s,b)$; changing endpoints does not change these time-integral norms.
For a stochastic interval, each norm is evaluated on that interval
for the fixed sample point before taking an expectation. These
abbreviations specify the solution and forcing norms. An additional
process norm will be defined for the stochastic contraction; all other
norm domains are displayed explicitly. We write $F(f)=|f|^4f$ and
\[
 b_0=\Bigl(\sum_{j=1}^m\|V_j\|_{L^\infty(\R^3)}^2\Bigr)^{1/2}.
\]
Constants denoted by $C_V$ depend only on $b_0,b_1$ and universal
Strichartz constants. The letter $C$ denotes a finite numerical constant
whose value may increase from one estimate to the next; dependencies
on further parameters are indicated. We write $\mathbf1_A$ for the
indicator of $A$, $s\wedge t=\min\{s,t\}$ and
$s\vee t=\max\{s,t\}$. The notation $\mathcal F_\sigma$ denotes the
sigma-field at a stopping time $\sigma$, and
$\inf\varnothing=\infty$ in stopping-time definitions.
Interval endpoints in stochastic integrals are
interpreted as $(s,t]$, so that restarting at time $s$ uses the
post-jump datum $u(s)$ and does not repeat the jump at $s$.

\section{Deterministic and stochastic estimates}\label{sec:analytic}

The local construction and the continuation argument use estimates
whose constants are uniform over time intervals, including intervals
approaching a finite lifetime and the half-line. The deterministic
estimates below place four factors of the quintic nonlinearity in the
scaling-invariant spacetime space and the remaining factor in an
admissible Strichartz space. The stochastic estimate uses the same
spacetime norm as a martingale-type-$2$ Banach norm. Its probability
moment remains two, independently of the time integrability exponent.
All linear constants in this section are independent of the length
and location of the time interval.

\begin{lemma}\label{lem:strichartz}
Let $I=[s,b]$, where $b$ may be infinite. Then
\begin{align}
 \|S(\cdot-s)f\|_{\mathcal X(I)}
     &\le C\|f\|_{H^1(\R^3)},\label{str:free}\\
 \left\|\int_s^t S(t-r)h(r)\dd r\right\|_{\mathcal X(I)}
     &\le C\|h\|_{\mathcal N(I)},\label{str:dual}\\
 \left\|\int_s^t S(t-r)k(r)\dd r\right\|_{\mathcal X(I)}
     &\le C\|k\|_{L^1(I;H^1(\R^3))}.\label{str:lone}
\end{align}
Moreover,
\begin{equation}\label{str:sobolev}
 \|f\|_{L^{10}(I\times\R^3)}\le C\|f\|_{Y(I)}.
\end{equation}
\end{lemma}

\begin{proof}
The homogeneous Strichartz estimates \cite{KeelTao1998}, with
admissible pairs $(\infty,2)$ and $(10,30/13)$, apply to $f$ and its
first derivatives. Commutation of the gradient with the free group
then gives \eqref{str:free}. For \eqref{str:dual}, use the
retarded inhomogeneous estimates with the same output pairs and the
dual endpoint forcing space $L^2(I;L^{6/5}(\R^3))$. Applied to
$\nabla h$, they yield
\[
 \begin{split}
 &\sup_{t\in I}\left\|\int_s^t S(t-r)\nabla h(r)\dd r
                                             \right\|_{L^2(\R^3)}\\
 &\quad+\left\|\int_s^t S(t-r)\nabla h(r)\dd r
                      \right\|_{L^{10}(I;L^{30/13}(\R^3))}
       \le C\|\nabla h\|_{L^2(I;L^{6/5}(\R^3))}.
 \end{split}
\]
The same estimate without $\nabla$ controls the zeroth-order term;
together they give the inhomogeneous Sobolev norm in
\eqref{str:dual}.

For \eqref{str:lone}, apply the homogeneous estimate at the integration
time $r$ and then use Minkowski's inequality:
\[
 \begin{split}
 \left\|\int_s^tS(t-r)k(r)\dd r\right\|_{Y(I)}
 &\le\int_s^b
       \|\mathbf1_{[r,b]}S(\cdot-r)k(r)\|_{Y(I)}\dd r\\
 &\le C\int_s^b\|k(r)\|_{H^1(\R^3)}\dd r.
 \end{split}
\]
The supremum in $H^1$ is bounded by the same integral using unitarity.
Finally, the spatial Sobolev embedding
$W^{1,30/13}(\R^3)\hookrightarrow L^{10}(\R^3)$ gives
\eqref{str:sobolev}. The Duhamel terms in \eqref{str:dual} and
\eqref{str:lone} are continuous in $H^1$ by approximation with smooth
functions and the displayed bounds. We shall also use the dual
homogeneous estimate
\begin{equation}\label{str:dualhomogeneous}
 \left\|\int_c^d S(-r)h(r)\dd r\right\|_{H^1(\R^3)}
                  \le C\|h\|_{\mathcal N([c,d])}.
\end{equation}
For its zeroth-order part, pair the integral against $g\in L^2(\R^3)$,
use H\"older's inequality in $L^2([c,d];L^{6/5}(\R^3))$ and
$L^2([c,d];L^6(\R^3))$, and
then apply homogeneous Strichartz to $S(r)g$. Apply the same argument
to each spatial derivative for the $H^1$ estimate. This interpretation
defines the integral even when $h$ is not integrable as an $H^1$-valued
function. In particular, for $h\in\mathcal N([s,b])$, its integral over
intervals shrinking to either endpoint tends to zero in $H^1$.
\end{proof}

\begin{lemma}\label{lem:quintic}
For $u,v\in Y(I)$,
\begin{align}
 \|F(u)\|_{\mathcal N(I)}&\le C\|u\|_{Y(I)}^5,
                                                    \label{F:single}\\
 \|F(u)-F(v)\|_{\mathcal N(I)}
  &\le C\bigl(\|u\|_{Y(I)}+\|v\|_{Y(I)}\bigr)^4
                          \|u-v\|_{Y(I)}.             \label{F:diff}
\end{align}
\end{lemma}

\begin{proof}
Since $F(w)=w^3\overline w^{\,2}$, differentiation of a smooth function
gives the explicit identity
\begin{equation}\label{F:chain-explicit}
 \nabla F(u)=3u^2\overline u^{\,2}\nabla u
                   +2u^3\overline u\,\nabla\overline u.
\end{equation}
Thus $|\nabla F(u)|\le5|u|^4|\nabla u|$. Applying the factorization
$a^k-b^k=(a-b)\sum_{\ell=0}^{k-1}a^{k-1-\ell}b^\ell$ to the two
polynomial factors yields
\[
 |F(u)-F(v)|\le C(|u|+|v|)^4|u-v|
\]
and, by subtracting \eqref{F:chain-explicit},
\[
 \begin{split}
 |\nabla(F(u)-F(v))|
 \le{}&C(|u|+|v|)^4|\nabla(u-v)|\\
      &+C(|u|+|v|)^3(|\nabla u|+|\nabla v|)|u-v|.
 \end{split}
\]
For the first summand of \eqref{F:chain-explicit}, the subtraction
separates the difference of the differentiated factors from that of
their coefficients:
\[
 u^2\overline u^{\,2}\nabla(u-v)
 +(u^2\overline u^{\,2}-v^2\overline v^{\,2})\nabla v.
\]
The coefficient difference is bounded by
$C(|u|+|v|)^3|u-v|$. The second summand is handled by the same
factorization, with $u^3\overline u-v^3\overline v$ in place of this
coefficient. This proves the displayed pointwise estimate.

H\"older's inequality in time and space, together with
\eqref{str:sobolev}, gives the required norm estimates. Set $w=u-v$.
For the first type of differentiated term, this gives
\[
 \begin{split}
 \||u|^4\nabla w\|_{L^2(I;L^{6/5}(\R^3))}
 &\le\|u\|_{L^{10}(I\times\R^3)}^4
                     \|\nabla w\|_{L^{10}(I;L^{30/13}(\R^3))}\\
 &\le C\|u\|_{Y(I)}^4\|w\|_{Y(I)}.
 \end{split}
\]
For the second type,
\[
 \begin{split}
 \||u|^3|w|\,|\nabla v|\|_{L^2(I;L^{6/5}(\R^3))}
 &\le\|u\|_{L^{10}(I\times\R^3)}^3
       \|w\|_{L^{10}(I\times\R^3)}
       \|\nabla v\|_{L^{10}(I;L^{30/13}(\R^3))}\\
 &\le C\|u\|_{Y(I)}^3\|v\|_{Y(I)}\|w\|_{Y(I)}.
 \end{split}
\]
The mixed monomials in $(|u|+|v|)^k$ satisfy the same estimate after
interchanging factors. For the undifferentiated difference, put $w$
in $L^{10}(I;L^{30/13}(\R^3))$ and the other four factors in
$L^{10}(I\times\R^3)$. Adding the resulting bounds proves
\eqref{F:diff}; taking $v=0$ proves \eqref{F:single}. Smooth
approximation in $Y(I)$ and the estimates themselves extend the
identities to Sobolev functions. The constants are independent of
$I$, so smallness in the local construction can be imposed through
the accumulated critical norm.
\end{proof}

The global deterministic theorem supplies the comparison solution
used in the stochastic continuation argument. We record its consequence
in the inhomogeneous norm $Y$.

\begin{lemma}\label{lem:detglobal}
For each $K<\infty$ there is $C_0(K)<\infty$ such that every solution
$v$ of \eqref{det:eq} with $\|v(s)\|_{H^1(\R^3)}\le K$ is global and
satisfies
\begin{equation}\label{det:bound}
 \sup_{t\ge s}\|v(t)\|_{H^1(\R^3)}
       +\|v\|_{Y([s,\infty))}\le C_0(K).
\end{equation}
\end{lemma}

\begin{proof}
The global energy-critical theorem \cite{CKSTT2008} gives a finite
nondecreasing function $C_{\rm ec}$ such that
\[
 \|v\|_{L^{10}([s,\infty)\times\R^3)}
                  \le C_{\rm ec}(E(v(s))).
\]
This is the only large-data deterministic input. Sobolev embedding gives
$E(v(s))\le K^2/2+C K^6$, while conservation gives
\[
 \sup_{t\ge s}\|v(t)\|_{H^1(\R^3)}^2
             \le K^2+2E(v(s)).
\]
Partition $[s,\infty)$ by successive levels of the integral of
$\|v(t)\|_{L^{10}(\R^3)}^{10}$, so that every interval except
possibly the last has $L^{10}$ spacetime norm exactly $\delta$.
The number $J$ of intervals then satisfies
\[
 J\le1+\delta^{-10}C_{\rm ec}(K^2/2+C K^6)^{10}.
\]
On each $I_j$, Strichartz and \eqref{F:chain-explicit} give
\[
 \|\nabla v\|_{L^{10}(I_j;L^{30/13}(\R^3))}
 \le C\|\nabla v(\inf I_j)\|_{L^2(\R^3)}
       +C\delta^4
          \|\nabla v\|_{L^{10}(I_j;L^{30/13}(\R^3))}.
\]
Choose $\delta$ so that the last coefficient is at most $1/2$.
The analogous undifferentiated estimate is
\[
 \|v\|_{L^{10}(I_j;L^{30/13}(\R^3))}
 \le C\|v(\inf I_j)\|_{L^2(\R^3)}
        +C\delta^4\|v\|_{L^{10}(I_j;L^{30/13}(\R^3))}.
\]
Mass conservation controls its initial term. After absorption, each
interval satisfies
\[
 \|v\|_{Y(I_j)}\le
       C\bigl(M(v(s))^{1/2}+E(v(s))^{1/2}\bigr).
\]
To justify absorption without presupposing a finite global $Y$ norm,
apply these estimates first on compact local intervals and increase
their right endpoints. Summing the tenth powers yields
\[
 \|v\|_{Y([s,\infty))}
 \le C J^{1/10}\bigl(K+(K^2+C K^6)^{1/2}\bigr).
\]
This specifies a finite function $C_0(K)$ in \eqref{det:bound},
depending only on $K$ and the universal deterministic constants.
In particular, it is independent of the starting time.
\end{proof}

\begin{lemma}\label{lem:stability}
For every $K<\infty$ there are $\delta_*(K)>0$ and $C_*(K)<\infty$
with the following property. Let $I=[s,b]$, and let $v$ solve
\eqref{det:eq} with $\|v(s)\|_{H^1(\R^3)}\le K$. Suppose that
$u\in\mathcal X(I)$ and $\eta\in\mathcal X(I)$ satisfy
$\eta(s)=0$ and
\begin{equation}\label{forced:mild}
 u(t)=S(t-s)u(s)-i\int_s^t S(t-r)F(u(r))\dd r+\eta(t).
\end{equation}
If
\[
 e:=\|u(s)-v(s)\|_{H^1(\R^3)}
           +\|\eta\|_{\mathcal X(I)}\le\delta_*(K),
\]
then
\begin{equation}\label{stability:bound}
 \|u-v\|_{\mathcal X(I)}\le C_*(K)e,
 \qquad \|u\|_{Y(I)}\le C_0(K)+C_*(K)e.
\end{equation}
The constants do not depend on $b$, and the assertion applies to
c\`adl\`ag external functions $\eta$.
\end{lemma}

\begin{proof}
We first establish an estimate on a single interval $I_j=[t_j,t_{j+1}]$
such that $\|v\|_{Y(I_j)}\le\gamma$, where $\gamma$ will be fixed
below. Let $w=u-v$ and reset the external function by
\begin{equation}\label{stability:reset}
 \eta_j(t)=\eta(t)-S(t-t_j)\eta(t_j),\qquad t\in I_j.
\end{equation}
Then $\eta_j(t_j)=0$. If $C_S$ is the constant in
\eqref{str:free}, restriction to $I_j$ gives
\begin{equation}\label{stability:resetbound}
 \|\eta_j\|_{\mathcal X(I_j)}
 \le\|\eta\|_{\mathcal X(I_j)}
         +C_S\|\eta(t_j)\|_{H^1(\R^3)}
 \le (1+C_S)\|\eta\|_{\mathcal X(I)}.
\end{equation}
The two mild equations imply
\[
 w(t)=S(t-t_j)w(t_j)
       -i\int_{t_j}^t S(t-r)(F(u(r))-F(v(r)))\dd r+\eta_j(t).
\]
In particular, $z_j=w-\eta_j$ is continuous in $H^1(\R^3)$.
Its cumulative $Y$ norm is continuous in the upper endpoint as well.
For $c\in[t_j,t_{j+1}]$, put
\[
 Z_j(c)=\|z_j\|_{\mathcal X([t_j,c])},\qquad
 d_j=\|w(t_j)\|_{H^1(\R^3)}+\|\eta_j\|_{\mathcal X(I_j)}.
\]
The function $Z_j$ is continuous, including at $t_j$, where
$Z_j(t_j)=\|w(t_j)\|_{H^1(\R^3)}$.

The difference estimate \eqref{F:diff}, together with
$u=v+z_j+\eta_j$, yields
\begin{equation}\label{stability:short}
 Z_j(c)\le C_S d_j
       +C\bigl(2\gamma+Z_j(c)+d_j\bigr)^4
                                           \bigl(Z_j(c)+d_j\bigr).
\end{equation}
Enlarge $C_S$ to be at least one and choose $\gamma>0$ such that
$C(3\gamma)^4\le1/8$. Choose $d_0>0$ so small that
$(4C_S+1)d_0\le\gamma$. Under the bootstrap bound
$Z_j(c)\le4C_Sd_j$, with $d_j\le d_0$, equation
\eqref{stability:short} gives
\[
 Z_j(c)\le C_Sd_j+\frac18(Z_j(c)+d_j)
        \le\frac87(C_S+1/8)d_j\le2C_Sd_j.
\]
For $d_j>0$, this improves the bootstrap bound. Since the initial
value satisfies it and $Z_j$ is continuous, the estimate holds up to
$t_{j+1}$. When $d_j=0$, the same conclusion follows by replacing it on the
right of \eqref{stability:short} by a positive $\eps\le d_0$ and
then letting $\eps\downarrow0$. We have proved
\begin{equation}\label{stability:oneinterval}
 \|w\|_{\mathcal X(I_j)}
   \le(2C_S+1)d_j
   \le C_1\bigl(\|w(t_j)\|_{H^1(\R^3)}
                      +\|\eta\|_{\mathcal X(I)}\bigr)
\end{equation}
whenever $d_j\le d_0$, with an absolute constant $C_1\ge2$.
Subtracting $\eta_j$ removes every jump from the quantity used in
the continuity argument; no continuity of $u$ or $\eta$ is required.

We now choose the partition for the whole interval. By
Lemma~\ref{lem:detglobal}, successive levels of the tenth-power
$Y$ integral of $v$ give at most
\[
 J_K=\left\lceil1+(C_0(K)/\gamma)^{10}\right\rceil
\]
intervals, independently of $b$. Write
$e_j=\|w(t_j)\|_{H^1(\R^3)}$ and
$h=\|\eta\|_{\mathcal X(I)}$. Estimate
\eqref{stability:oneinterval} implies
\[
 e_{j+1}+h\le(C_1+1)(e_j+h),
\]
provided $e_j+(1+C_S)h\le d_0$. Since $e_0+h=e$, induction
gives
\[
 e_j+h\le(C_1+1)^j e.
\]
It is therefore enough to require
\[
 e\le\delta_*(K):=
   \frac{d_0}{(1+C_S)(C_1+1)^{J_K}}.
\]
This choice verifies the smallness hypothesis at every step.
The $H^1$ supremum over $I$ is the
maximum of the interval suprema, while the tenth powers of the
$Y$ norms add. Consequently,
\[
 \|w\|_{\mathcal X(I)}
       \le C_1(1+J_K^{1/10})(C_1+1)^{J_K} e.
\]
This proves the first bound in \eqref{stability:bound}; the second
follows from \eqref{det:bound} and the triangle inequality. For an
infinite right endpoint, apply these estimates on every finite
initial interval with the same $J_K$ and constants, and then increase
the endpoint. The lemma is pathwise: its comparison intervals need
not be stopping times when it is applied to a fixed sample path.
\end{proof}

\begin{lemma}\label{lem:maps}
For every $f\in H^1(\R^3)$ and $(t,z)$,
\begin{align}
 \|G_{t,z}f\|_{H^1(\R^3)}
       &\le C_V|a(t)||z|\|f\|_{H^1(\R^3)},\label{G:bound}\\
 \|R_{t,z}f\|_{H^1(\R^3)}
       &\le C_V a(t)^2|z|^2\|f\|_{H^1(\R^3)},\label{R:bound}\\
 \|B(t)f\|_{H^1(\R^3)}
       &\le C_V Q a(t)^2\|f\|_{H^1(\R^3)}.\label{B:bound}
\end{align}
The same bounds hold for differences, with $f$ replaced by $f-g$.
\end{lemma}

\begin{proof}
Put $\theta=a(t)A_z$. The integral representations
\[
 e^{-i\theta}-1=-i\theta\int_0^1 e^{-ir\theta}\dd r,
 \qquad
 e^{-i\theta}-1+i\theta
       =-\theta^2\int_0^1(1-r)e^{-ir\theta}\dd r
\]
give, for all real $\theta$,
\[
 |e^{-i\theta}-1|\le|\theta|,
 \qquad |e^{-i\theta}-1+i\theta|\le\tfrac12|\theta|^2.
\]
Also $\|\theta\|_{L^\infty(\R^3)}\le|a(t)||z|b_0$ and
$\|\nabla\theta\|_{L^\infty(\R^3)}\le|a(t)||z|b_1$.
The product rule gives
\[
 \nabla(G_{t,z}f)=(e^{-i\theta}-1)\nabla f
                         -ie^{-i\theta}f\nabla\theta.
\]
The two parts of the Sobolev norm therefore satisfy
\[
 \begin{split}
 \|G_{t,z}f\|_{L^2(\R^3)}
     &\le |a(t)||z|b_0\|f\|_{L^2(\R^3)},\\
 \|\nabla(G_{t,z}f)\|_{L^2(\R^3)}
     &\le |a(t)||z|
        \bigl(b_0\|\nabla f\|_{L^2(\R^3)}
                         +b_1\|f\|_{L^2(\R^3)}\bigr).
 \end{split}
\]
This proves \eqref{G:bound}. For the remainder,
\[
 \nabla(e^{-i\theta}-1+i\theta)
              =-i(e^{-i\theta}-1)\nabla\theta,
\]
and the product rule gives
\[
 \begin{split}
 \|R_{t,z}f\|_{L^2(\R^3)}
 &\le\tfrac12a(t)^2|z|^2b_0^2\|f\|_{L^2(\R^3)},\\
 \|\nabla(R_{t,z}f)\|_{L^2(\R^3)}
 &\le a(t)^2|z|^2
     \bigl(\tfrac12b_0^2\|\nabla f\|_{L^2(\R^3)}
                        +b_0b_1\|f\|_{L^2(\R^3)}\bigr).
 \end{split}
\]
This proves \eqref{R:bound}. For a fixed $f$, the map
$z\mapsto R_{t,z}f$ is strongly measurable in $H^1$: continuity in
$z$ follows from the displayed product formulas and dominated
convergence. Its norm is integrable by \eqref{noise:moment}.
Consequently the Bochner integral defining $B(t)f$ exists, and
\[
 \|B(t)f\|_{H^1(\R^3)}
 \le\int_Z\|R_{t,z}f\|_{H^1(\R^3)}\nu(\dd z)
 \le C_V Q a(t)^2\|f\|_{H^1(\R^3)}.
\]
Linearity gives the difference estimates. The scalar exponential
bounds hold at every amplitude, so these constants also apply to
unbounded jump sizes. They depend only on the profile bounds $b_0$
and $b_1$. The spatial product rules require only the weak first
derivatives of the bounded Lipschitz profiles.
\end{proof}

\begin{lemma}\label{lem:poisson}
Let $\Phi$ be predictable and $H^1(\R^3)$-valued, with
\[
 \E\int_s^b\int_Z\|\Phi(r,z)\|_{H^1(\R^3)}^2\nu(\dd z)\dd r<\infty.
\]
The convolution
$W(t)=\int_s^t\int_Z S(t-r)\Phi(r,z)\widetilde N(\dd r,\dd z)$
has a c\`adl\`ag $H^1$ version and satisfies
\begin{equation}\label{poisson:bound}
 \E\|W\|_{\mathcal X([s,b])}^2
 \le C\E\int_s^b\int_Z\|\Phi(r,z)\|_{H^1(\R^3)}^2\nu(\dd z)\dd r.
\end{equation}
The same bound holds for the convolution integrated over a stochastic
interval $(\sigma,t]$, with norm on $[\sigma,\tau]$, whenever
$\sigma\le\tau$ are stopping times and the right-hand side is finite.
The constant is independent of both stopping times.
\end{lemma}

\begin{proof}
Put $I=[s,b]$, first with $b<\infty$. We first consider bounded
predictable simple integrands supported on finitely many time intervals
and on mark sets of finite measure. Their convolutions are finite jump
sums minus Bochner compensator integrals, which identify the
Hilbert-space and spacetime-valued constructions. The estimates below
extend this identification by completion in the norm on the
right-hand side of \eqref{poisson:bound}.

The process
\[
 P(t)=\int_s^t\int_ZS(-r)\Phi(r,z)\widetilde N(\dd r,\dd z)
\]
is a square-integrable Hilbert-space martingale with c\`adl\`ag paths.
The isometry, Doob's inequality, and unitarity give
\[
 \E\sup_{s\le t\le b}\|S(t)P(t)\|_{H^1(\R^3)}^2
 \le4\E\int_s^b\int_Z\|\Phi(r,z)\|_{H^1(\R^3)}^2\nu(\dd z)\dd r.
\]
For the spacetime term, regard
\[
 (r,z)\longmapsto
       [t\longmapsto\mathbf1_{\{r\le t\le b\}}S(t-r)\Phi(r,z)]
\]
as a function with values in $Y(I)$. To justify its predictability,
define the deterministic operators
\[
 K_rf=\mathbf1_{[r,b]}S(\cdot-r)f,\qquad r\in[s,b].
\]
The homogeneous estimate gives $\|K_rf\|_{Y(I)}\le C\|f\|_{H^1(\R^3)}$
uniformly in $r$. For a fixed $f$ and $q\le r$, the difference is
\[
 K_rf-K_qf
 =\mathbf1_{[r,b]}S(\cdot-r)(f-S(r-q)f)
       -\mathbf1_{[q,r)}S(\cdot-q)f.
\]
The first term tends to zero in $Y(I)$ as $r\to q$ by strong
continuity of the free group on $H^1(\R^3)$. The second tends to zero
by absolute continuity of the time integral for the fixed free
trajectory. The reversed ordering of $q,r$ gives the same conclusion.
Thus $r\mapsto K_rf$ is continuous. The uniform operator bound also
shows joint continuity in $(r,f)$, so $K_r\Phi(r,z)$ is a predictable
$Y(I)$-valued integrand.

The Banach space $Y(I)$ has martingale type $2$. Indeed, the map
\[
 f\longmapsto(f,\partial_1f,\partial_2f,\partial_3f)
\]
embeds $Y(I)$ as a closed subspace of
$L^{10}(I;L^{30/13}(\R^3;\mathbb C^4))$, with an equivalent norm.
Both mixed-norm exponents are at least two, so this space has
martingale type $2$;
the property passes to the closed range. The constants are independent
of the time measure of $I$. The second-moment inequality for a
Poisson integral in such a space therefore gives
\[
 \E\|W\|_{Y([s,b])}^2
 \le C\E\int_s^b\int_Z
   \|\mathbf1_{[r,b]}S(\cdot-r)\Phi(r,z)\|_{Y([s,b])}^2
                      \nu(\dd z)\dd r.
\]
The homogeneous estimate bounds each deterministic kernel by
\[
 \|\mathbf1_{[r,b]}S(\cdot-r)\Phi(r,z)\|_{Y(I)}^2
                 \le C\|\Phi(r,z)\|_{H^1(\R^3)}^2.
\]
Here $r$ indexes the stochastic integration and $t$ is the time
coordinate in $Y(I)$. The type-$2$ inequality takes a second moment
in probability of the entire $Y(I)$ norm. Thus the time exponent ten
requires no additional moment of $\nu$.
The second-moment inequality used here follows from Theorem~2.11 in
\cite{DirksenMaasVanNeerven2013}, with both its probability exponent
and its martingale-type exponent equal to two. The integrability
conditions for other probability exponents are treated
in \cite{Dirksen2014}; the application to
Schr\"odinger convolutions is developed in
\cite{BrzezniakLiuZhu2021}.

To identify the two versions, approximate $\Phi$ by simple predictable
$\Phi_n$ in the norm
\[
 \left(\E\int_s^b\int_Z
          \|\Phi_n(r,z)-\Phi(r,z)\|_{H^1(\R^3)}^2
                                     \nu(\dd z)\dd r\right)^{1/2}.
\]
The first estimate gives convergence of the corresponding convolutions
in expected squared uniform $H^1$ norm. The type-$2$ estimate gives
convergence in expected squared $Y$ norm. A subsequence converges
almost surely in both norms; their limits agree as distributions on
space-time because they agree for each approximation. Thus the
c\`adl\`ag $H^1$ version represents the same spacetime distribution
as the $Y$-valued integral. Combining the two estimates proves
\eqref{poisson:bound}.

For the random-interval assertion, set
\[
 W(t)=\int_{(\sigma,t]}\int_Z S(t-r)\Phi(r,z)
 \widetilde N(\dd r,\dd z),\qquad \sigma\le t\le\tau,
\]
and extend the integrand by zero outside $(\sigma,\tau]$. The process
$\mathbf1_{\{\sigma<r\le\tau\}}$ is adapted and left continuous in
$r$, hence predictable. On a deterministic enclosing interval the
preceding proof applies without alteration. Restrict the resulting
convolution norm to $[\sigma,\tau]$ to obtain the random-interval
estimate. For $A\in\mathcal F_\sigma$, the process
$\mathbf1_A\mathbf1_{\{\sigma<r\le\tau\}}$ is also predictable.
Applying the same estimate with this indicator yields
\[
 \E\bigl[\mathbf1_A\|W\|_{\mathcal X([\sigma,\tau])}^2\bigr]
 \le C\E\left[\mathbf1_A\int_\sigma^\tau\int_Z
        \|\Phi(r,z)\|_{H^1(\R^3)}^2\nu(\dd z)\dd r\right].
\]
Since this holds for every $A\in\mathcal F_\sigma$, it also gives
the conditional estimate with respect to $\mathcal F_\sigma$.
For unbounded endpoints first stop at an integer
time and then increase that time. The integral and the norm bounds
pass to the limit by monotone convergence, with the same constants.
\end{proof}

We use this predictable interval convention throughout the local
construction. Later, the continuation argument also selects times
after fixing a sample path; those times enter only increments of
already constructed processes.

\begin{lemma}\label{lem:potential}
Let $P\in L^1(0,T;W^{1,\infty}(\R^3;\R))$. The deterministic equation
\begin{equation}\label{potential:eq}
 \partial_t w=i(\Delta w-F(w))+iP(t)w,
       \qquad w(0)\in H^1(\R^3),
\end{equation}
has a unique solution in $C([0,T];H^1(\R^3))\cap Y([0,T])$.
Its mass is constant, and
\begin{equation}\label{potential:energy}
 E(w(t))-E(w(s))
    =\int_s^t\int_{\R^3}\nabla P(r,x)\cdot j(w(r,x))\dd x\dd r.
\end{equation}
\end{lemma}

\begin{proof}
Local existence follows by treating the potential as an integrable
linear perturbation. Multiplication obeys
\begin{equation}\label{potential:product}
 \|P(t)f\|_{H^1(\R^3)}
       \le C\|P(t)\|_{W^{1,\infty}(\R^3)}\|f\|_{H^1(\R^3)},
\end{equation}
as follows by expanding $\nabla(Pf)=P\nabla f+f\nabla P$.
For $f=w(s)$, the local integral map is
\[
 \mathcal T_Pw(t)=S(t-s)f-i\int_s^tS(t-r)F(w(r))\dd r
                              +i\int_s^tS(t-r)P(r)w(r)\dd r.
\]
Choose a small absolute $\eta_0>0$, and set
$K_f=2C_S\|f\|_{H^1(\R^3)}+1$, where $C_S\ge1$ bounds the
constants in Lemma~\ref{lem:strichartz}. By absolute continuity of the free
$Y$ integral and the potential integral, an interval $I=[s,s+h]$
can be chosen so that
\[
 \|S(\cdot-s)f\|_{Y(I)}\le\eta_0/2,\qquad
 C K_f\int_I\|P(r)\|_{W^{1,\infty}(\R^3)}\dd r\le\eta_0/4.
\]
On the closed set
\[
 \left\{w\in C(I;H^1(\R^3))\cap Y(I):
     \sup_{t\in I}\|w(t)\|_{H^1(\R^3)}\le K_f,
     \quad\|w\|_{Y(I)}\le2\eta_0\right\},
\]
the nonlinear term is bounded in $\mathcal X(I)$ by
$C(2\eta_0)^5$, and the potential term by the second displayed
bound. Decrease $\eta_0$ so that $C(2\eta_0)^5\le\eta_0/4$ and the
difference constant $C(4\eta_0)^4$ is at most $1/4$.
If necessary decrease $h$ once more so that the potential difference
constant is also at most $1/4$. The map preserves the indicated set
and is a contraction in $\mathcal X(I)$. This proves local existence
and uniqueness, with continuous $H^1$ paths.

For smooth coefficients and data, direct differentiation of mass gives
\[
 \frac{\mathrm d}{\mathrm dt}M(w(t))
   =2\operatorname{Re}\int_{\R^3}\overline w
                         (i\Delta w-i|w|^4w+iPw)\dd x=0.
\]
The Laplacian term is integrated by parts and all three resulting
integrals are purely imaginary. For energy, its real derivative is
\[
 E'(w)h=\operatorname{Re}\int_{\R^3}
                    \nabla\overline w\cdot\nabla h
                              +|w|^4\overline w h\dd x.
\]
Putting $h=i(\Delta w-F(w))$ gives zero: after one integration by
parts the integrand is the real part of
$-i\overline{(\Delta w-F(w))}(\Delta w-F(w))$.
The remaining derivative is
\[
 \operatorname{Re}\int_{\R^3}
       \nabla\overline w\cdot\nabla(iPw)+|w|^4\overline w(iPw)\dd x
       =\int_{\R^3}\nabla P\cdot j(w)\dd x.
\]
The terms containing $P|\nabla w|^2$ and $P|w|^6$ are purely imaginary.
This proves \eqref{potential:energy} for smooth solutions.

To justify the identities for $H^1$ data and the stated coefficient
regularity, we approximate the potential in two stages.
First suppose $P(t,x)=c(t)V(x)$, with $c\in L^1(0,T)$ and
$V\in W^{1,\infty}(\R^3;\R)$. This is the form needed for the
finite-activity equations. Let $V_\ell$ be real spatial mollifications.
Then
\[
 \sup_\ell\|V_\ell\|_{W^{1,\infty}(\R^3)}<\infty,
 \quad V_\ell\to V,\quad\nabla V_\ell\to\nabla V
                       \quad\text{almost everywhere}.
\]
For every fixed $f\in H^1(\R^3)$,
\begin{equation}\label{potential:strongmultiplier}
 \|(V_\ell-V)f\|_{H^1(\R^3)}\longrightarrow0.
\end{equation}
Indeed, the three terms in this norm are controlled by
$(V_\ell-V)f$, $(V_\ell-V)\nabla f$ and
$(\nabla V_\ell-\nabla V)f$ in $L^2$, and each converges by dominated
convergence. Uniform boundedness of the multipliers makes the
convergence uniform on compact subsets of $H^1$: take a finite net
of such a compact subset, use convergence at the net points, and then
use the common operator bound on the net error.

Approximate $c$ in $L^1(0,T)$ by smooth real $c_\ell$ and $f$ in
$H^1$ by smooth $f_\ell$. On a compact interval in the lifetime of
the local solution, its continuous $H^1$ trajectory is compact.
It follows from \eqref{potential:strongmultiplier} that
\begin{equation}\label{potential:forcing-convergence}
 \|(c_\ell V_\ell-cV)w\|_{L^1(I;H^1(\R^3))}\longrightarrow0.
\end{equation}
More explicitly, this norm is at most
\[
 \begin{split}
 &C\|c_\ell-c\|_{L^1(I)}
                   \sup_{t\in I}\|w(t)\|_{H^1(\R^3)}\\
 &\qquad+\|c\|_{L^1(I)}
                   \sup_{t\in I}\|(V_\ell-V)w(t)\|_{H^1(\R^3)},
 \end{split}
\]
where the constant uses the common multiplier bound. The first term
tends to zero by the approximation of $c$; the second tends to zero
by compactness of the trajectory and \eqref{potential:strongmultiplier}.

The required convergence of solutions follows from a local comparison.
Put $P_\ell=c_\ell V_\ell$
and partition a compact interval in the lifetime of $w$ into finitely
many intervals $I_j=[t_j,t_{j+1}]$ on which $\|w\|_{Y(I_j)}\le\gamma$
and the potential integrals are small. Such a partition can be used
for all sufficiently large $\ell$: the norms of $V_\ell$ are uniformly
bounded, and $c_\ell\to c$ in $L^1(0,T)$ makes the corresponding time
integrals uniformly absolutely continuous. Subtract the mild equations
and split the potential difference as
\[
 P_\ell w_\ell-Pw
       =P_\ell(w_\ell-w)+(P_\ell-P)w.
\]
On an initial part $[t_j,t]$ of $I_j$, the bootstrap assumption
$\|w_\ell-w\|_{Y([t_j,t])}\le\gamma$ implies
$\|w_\ell\|_{Y([t_j,t])}\le2\gamma$. Choose $\gamma$ and the
potential integrals so that the sum of the nonlinear and potential
difference constants is at most $1/2$. Lemma~\ref{lem:quintic} and
\eqref{potential:product} then give
\[
 \begin{split}
 \|w_\ell-w\|_{\mathcal X([t_j,t])}
 \le{}&C\|w_\ell(t_j)-w(t_j)\|_{H^1(\R^3)}
          +\tfrac12\|w_\ell-w\|_{\mathcal X([t_j,t])}\\
      &+C\|(P_\ell-P)w\|_{L^1(I_j;H^1(\R^3))}.
 \end{split}
\]
Absorbing the middle term gives twice the sum of the two error terms.
For sufficiently large $\ell$, this is smaller than $\gamma/2$,
which closes the bootstrap by continuity. If a regularized local
lifetime ended inside $I_j$ while these bounds held, its mild formula
would give an $H^1$ limit there: the quintic Duhamel integral is Cauchy
by \eqref{str:dualhomogeneous}, and the potential integral is absolutely
integrable in $H^1$. The local contraction from that limit would extend
the solution. Thus the comparison reaches $t_{j+1}$ and gives convergence
of the endpoint data. Induction over the finite partition proves
convergence in $C(I;H^1(\R^3))\cap Y(I)$.
For smooth regularizations, persistence of higher Sobolev regularity
justifies the differential identities above on these intervals;
only the $H^1$ and critical spacetime bounds are used in passing to
the limit. The deterministic persistence argument is given in
\cite{Cazenave2003,Tao2006}.

For a general real $P\in L^1(0,T;W^{1,\infty}(\R^3))$, approximate first
by finite sums of time-simple functions with real $W^{1,\infty}$
profiles in this Bochner norm. Apply the preceding spatial argument
to each profile. The product estimate bounds the error of the first
approximation by its $L^1(0,T;W^{1,\infty}(\R^3))$ norm times the bounded
$H^1$ trajectory. A diagonal approximation then gives the same local
convergence. Thus the argument covers the full statement of the lemma.

To pass to the limit in the energy formula, use
\[
 \|j(f)-j(g)\|_{L^1(\R^3)}
 \le\|f-g\|_{L^2(\R^3)}\|\nabla f\|_{L^2(\R^3)}
       +\|g\|_{L^2(\R^3)}\|\nabla(f-g)\|_{L^2(\R^3)}.
\]
Hence the currents of the regularized solutions converge uniformly
in $L^1(\R^3)$ on every compact local interval. For $P=cV$, the
difference of the integrated currents has the following three terms:
\[
 \begin{split}
 &\int_I\int_{\R^3}
       \bigl(\nabla P_\ell\cdot j(w_\ell)-\nabla P\cdot j(w)\bigr)
                                      \dd x\dd t\\
 &\quad=\int_I\int_{\R^3}
       c_\ell\nabla V_\ell\cdot\bigl(j(w_\ell)-j(w)\bigr)\dd x\dd t\\
 &\qquad+\int_I\int_{\R^3}
       (c_\ell-c)\nabla V_\ell\cdot j(w)\dd x\dd t\\
 &\qquad+\int_I\int_{\R^3}
       c(\nabla V_\ell-\nabla V)\cdot j(w)\dd x\dd t.
 \end{split}
\]
The absolute values of the first two terms are bounded, respectively, by
\[
 \begin{gathered}
 C\|c_\ell\|_{L^1(I)}
       \sup_{t\in I}\|j(w_\ell(t))-j(w(t))\|_{L^1(\R^3)},\\
 C\|c_\ell-c\|_{L^1(I)}
       \sup_{t\in I}\|j(w(t))\|_{L^1(\R^3)}.
 \end{gathered}
\]
Both tend to zero. In the third term, dominated convergence applies
with the fixed integrable majorant
$2\|\nabla V\|_{L^\infty(\R^3)}|c(t)|\,|j(w(t,x))|$.
The temporal coefficients converge in $L^1$, while the spatial
derivatives converge almost everywhere under a uniform bound. For a
finite sum of separated potentials, apply this argument term by term. The
remaining general-potential approximation is controlled in its Bochner
$L^1(0,T;W^{1,\infty}(\R^3))$ norm, using
\[
 \left|\int_{\R^3}\nabla P\cdot j(w)\dd x\right|
 \le \|\nabla P\|_{L^\infty(\R^3)}
       \|w\|_{L^2(\R^3)}\|\nabla w\|_{L^2(\R^3)}.
\]

The energy functional is continuous under the uniform $H^1$ convergence
just obtained. Thus both endpoints of the energy formula also
converge, proving \eqref{potential:energy}. The mass identity passes
to the limit by the same argument.

The energy identity yields an a priori bound independent of the
local critical norm. With $M=M(w(0))$, it implies, almost everywhere,
\[
 \left|\frac{\mathrm d}{\mathrm dt}E(w(t))\right|
 \le\sqrt{2M E(w(t))}\,
                         \|\nabla P(t)\|_{L^\infty(\R^3)}.
\]
For $\eps>0$, divide by $2\sqrt{E(w(t))+\eps}$ and integrate.
Since $\sqrt{E/(E+\eps)}\le1$, letting $\eps\downarrow0$ yields
\[
 \sqrt{E(w(t))}\le\sqrt{E(w(0))}
       +\sqrt{M/2}\int_0^t\|\nabla P(r)\|_{L^\infty(\R^3)}\dd r.
\]
Together with mass conservation, this bounds $w$ in $H^1$ up to any
possible finite maximal time $\tau\le T$. Denote that bound by $K$.
The product estimate gives
\[
 \|Pw\|_{L^1(0,\tau;H^1(\R^3))}<\infty.
\]
More precisely, by \eqref{str:lone},
\[
 \begin{split}
 \left\|i\int_s^tS(t-r)P(r)w(r)\dd r
                               \right\|_{\mathcal X([s,\tau))}\\
 \le CK\int_s^\tau\|P(r)\|_{W^{1,\infty}(\R^3)}\dd r
                        \longrightarrow0\quad(s\uparrow\tau).
 \end{split}
\]
Choose $s$ so that this quantity is at most $\delta_*(K)$ and compare
on $[s,b]$, $b<\tau$, with the unforced solution from $w(s)$.
Lemma~\ref{lem:stability} gives a bound for $\|w\|_{Y([s,b])}$
independent of $b$. Increase $b$ to conclude
$\|w\|_{Y([0,\tau))}<\infty$.

Finite critical norm supplies the endpoint datum needed for
continuation. In the interaction representation,
\[
 S(-t)w(t)=w(0)-i\int_0^tS(-r)F(w(r))\dd r
                           +i\int_0^tS(-r)P(r)w(r)\dd r.
\]
The first integral has an $H^1$ limit as $t\uparrow\tau$ by
\eqref{F:single}, \eqref{str:dualhomogeneous}, and the finite $Y$
norm. The second has a limit by its absolute $H^1$ integrability.
Thus $w(t)$ has an $H^1$ limit at $\tau$. Starting the local
contraction from this limit extends the solution if $\tau<T$, and
includes the endpoint if $\tau=T$. The continuation therefore follows
from the energy bound together with the deterministic stability
estimate and finite critical norm.
\end{proof}

\section{Local solutions and finite-activity approximation}\label{sec:local}

The construction uses two continuous accumulated quantities: the
noise intensity and the critical spacetime norm. They determine
stopping times on which the mild equation is contractive in a
second-moment process norm. The same estimates give local convergence
of the finite-activity equations before any energy identity is passed
to the limit. Energy estimates enter only in the next section.

Choose a nonincreasing Lipschitz function
$\chi:[0,\infty)\to[0,1]$ equal to one on $[0,1]$ and zero on
$[2,\infty)$. For $\eta>0$ and an interval starting at $s$, put
\begin{equation}\label{cutoff:def}
 \mathcal F_{s,\eta}(u)(t)
   =\chi\bigl(\eta^{-1}\|u\|_{Y([s,t])}\bigr)F(u(t)).
\end{equation}
For an adapted c\`adl\`ag $H^1$ process belonging to $Y(I)$, the
accumulated norm is continuous and adapted, and the cutoff is therefore
predictable. Measurability follows by viewing the spatial Sobolev norm
as an extended Borel function on $H^1$ and integrating in time.
The cutoff has no spatial dependence, so spatial differentiation
produces no additional terms.

\begin{lemma}\label{lem:cutoff}
For every interval $I=[s,b]$ and $u,v\in Y(I)$,
\begin{equation}\label{cutoff:lip}
 \|\mathcal F_{s,\eta}(u)-\mathcal F_{s,\eta}(v)\|_{\mathcal N(I)}
           \le C\eta^4\|u-v\|_{Y(I)}.
\end{equation}
In addition,
\begin{equation}\label{cutoff:single}
 \|\mathcal F_{s,\eta}(u)\|_{\mathcal N(I)}
      \le C\eta^5.
\end{equation}
The constants depend only on $\chi$ and the constants in
Lemma~\ref{lem:quintic}.
\end{lemma}

\begin{proof}
Write $A_u(t)=\|u\|_{Y([s,t])}$ and define $A_v$ in the same way.
Minkowski's inequality gives, for every $t\in I$,
\begin{equation}\label{cutoff:cumulative}
 |A_u(t)-A_v(t)|
       \le\|u-v\|_{Y([s,t])}\le\|u-v\|_{Y(I)}.
\end{equation}
If $E_u=\{t\in I:A_u(t)<2\eta\}$, monotonicity and continuity of
$A_u$ imply
\[
 \|\mathbf1_{E_u}u\|_{Y(I)}^{10}
     =\int_{E_u}\|u(t)\|_{W^{1,30/13}(\R^3)}^{10}\dd t
     \le(2\eta)^{10}.
\]
The same conclusion holds with $E_u$ replaced by $\{A_u\le2\eta\}$,
since the cumulative integral is constant on any interval where
$A_u=2\eta$.
The support of $\mathcal F_{s,\eta}(u)$ is contained in $E_u$, up to
such an endpoint. Equation~\eqref{F:single} therefore proves
\eqref{cutoff:single}.

For the difference estimate, divide $I$ into the measurable sets
$E=\{A_u\le A_v\}$ and $I\setminus E$. On $E$ we use the identity
\begin{equation}\label{cutoff:split}
 \mathcal F_{s,\eta}(u)-\mathcal F_{s,\eta}(v)
  =\chi(A_v/\eta)(F(u)-F(v))
   +\bigl(\chi(A_u/\eta)-\chi(A_v/\eta)\bigr)F(u).
\end{equation}
In the first term both accumulated norms are at most $2\eta$
wherever its coefficient is nonzero. Denote this smaller time set by
$E_1=E\cap\{A_v<2\eta\}$. Then
\[
 \|\mathbf1_{E_1}u\|_{Y(I)}
       +\|\mathbf1_{E_1}v\|_{Y(I)}\le4\eta.
\]
For the differentiated part of the first term, the pointwise bounds
in Lemma~\ref{lem:quintic} reduce the estimate to the
$L^2(I;L^{6/5}(\R^3))$ norms of
\[
 \mathbf1_{E_1}(|u|+|v|)^4|\nabla(u-v)|,\qquad
 \mathbf1_{E_1}(|u|+|v|)^3
               (|\nabla u|+|\nabla v|)|u-v|.
\]
H\"older's inequality places four undifferentiated factors in
$L^{10}(I\times\R^3)$ and the differentiated factor in
$L^{10}(I;L^{30/13}(\R^3))$. The four factors belonging to $u$ or $v$
on $E_1$ contribute at most $C\eta^4$; the difference factor contributes
$C\|u-v\|_{Y(I)}$. The undifferentiated expression is estimated
with the same exponents. Consequently,
\[
 \|\mathbf1_E\chi(A_v/\eta)(F(u)-F(v))\|_{\mathcal N(I)}
       \le C\eta^4\|u-v\|_{Y(I)}.
\]

For the second term in \eqref{cutoff:split}, a nonzero coefficient on
$E$ implies $A_u<2\eta$. The Lipschitz property of $\chi$ and
\eqref{cutoff:cumulative} give
\[
 \sup_{t\in I}|\chi(A_u(t)/\eta)-\chi(A_v(t)/\eta)|
       \le C\eta^{-1}\|u-v\|_{Y(I)}.
\]
Restriction to $\{A_u<2\eta\}$ and \eqref{F:single} now yield
\[
 \begin{split}
 &\|\mathbf1_E(\chi(A_u/\eta)-\chi(A_v/\eta))F(u)\|_{\mathcal N(I)}
 \\
 &\hspace{20mm}\le
 C\eta^{-1}\|u-v\|_{Y(I)}(2\eta)^5
 \le C\eta^4\|u-v\|_{Y(I)}.
 \end{split}
\]
On $I\setminus E$, interchange $u$ and $v$ in this decomposition.
Adding the two time restrictions proves \eqref{cutoff:lip}.
Only spatial derivatives enter $\mathcal N(I)$, so the measurable
time restrictions introduce no additional terms.
\end{proof}

\begin{proposition}\label{prop:local}
There is a pathwise unique maximal local mild solution $(u,\tau^*)$
of \eqref{main:mild}. It belongs to $\mathcal X([0,t])$ for every
$t<\tau^*$, almost surely. Its lifetime is announced by an increasing
sequence of stopping times $\sigma_k$, and
\begin{equation}\label{blowup:alternative}
 \{\tau^*<\infty\}\subset
 \{\|u\|_{Y([0,\tau^*))}=\infty\}
 \quad\text{up to a null set}.
\end{equation}
The construction can be restarted at a stopping time $\sigma$ from
an $\mathcal F_\sigma$-measurable, almost surely finite $H^1$ datum.
\end{proposition}

\begin{proof}
We begin with a deterministic interval $I=[s,b]$ and an
$\mathcal F_s$-measurable datum $f$ such that
$\E\|f\|_{H^1(\R^3)}^2<\infty$. The process space for the contraction
consists of adapted c\`adl\`ag processes with norm
\begin{equation}\label{local:processnorm}
 \|u\|_{\mathscr B(I)}
    =\bigl(\E\|u\|_{\mathcal X(I)}^2\bigr)^{1/2}.
\end{equation}
Here $\mathcal X(I)$ is the pathwise norm, whereas
$\mathscr B(I)$ includes the second moment in probability. Processes
are identified up to indistinguishability. Both terms in the pathwise
norm are measurable: the supremum can be taken over a countable dense
set together with the endpoints, and the spacetime term is a measurable
time integral. Thus this definition does not require Bochner strong
measurability in the uniform topology of all c\`adl\`ag paths.

For completeness, from a Cauchy sequence
choose a subsequence $(u_\ell)$ with
$\sum_\ell\|u_{\ell+1}-u_\ell\|_{\mathscr B(I)}<\infty$.
Minkowski's inequality implies that
$\sum_\ell\|u_{\ell+1}-u_\ell\|_{\mathcal X(I)}<\infty$
almost surely. Hence this subsequence converges uniformly in $H^1$
and also in $Y(I)$, on a common set of probability one. Uniform
limits of c\`adl\`ag functions are c\`adl\`ag. The limit is adapted,
and the two limits agree as functions for almost every time.
Fatou's lemma applied to differences gives convergence in
\eqref{local:processnorm}. This proves completeness. Left limits of
these adapted processes are predictable; these are the values used
in the Poisson integrals.

Replace $F$ by $\mathcal F_{s,\eta}$ in the mild equation and define
\[
 \begin{split}
 (\mathcal Tu)(t)={}&S(t-s)f
   -i\int_s^tS(t-r)\mathcal F_{s,\eta}(u)(r)\dd r\\
 &+\int_{(s,t]}\int_ZS(t-r)G_{r,z}u(r-)\widetilde N(\dd r,\dd z)
   \\
 &+\int_s^tS(t-r)B(r)u(r)\dd r.
 \end{split}
\]
Denote the nonlinear Duhamel term, the Poisson convolution and the
compensator term in this formula by $\mathcal T_Fu$, $\mathcal T_Gu$
and $\mathcal T_Bu$, respectively. Thus
$\mathcal Tu=S(\cdot-s)f+\mathcal T_Fu+\mathcal T_Gu+\mathcal T_Bu$;
the factor $-i$ is included in $\mathcal T_F$.
Put $q_I=Q\int_s^b a(r)^2\dd r$. Lemma~\ref{lem:cutoff} and the
deterministic Strichartz estimate give
\[
 \begin{split}
 \E\|\mathcal T_Fu-\mathcal T_Fv\|_{\mathcal X(I)}^2
 &\le C\E
  \|\mathcal F_{s,\eta}(u)-\mathcal F_{s,\eta}(v)\|_{\mathcal N(I)}^2\\
 &\le C\eta^8\E\|u-v\|_{Y(I)}^2.
 \end{split}
\]
For the stochastic term, Lemma~\ref{lem:poisson} and
\eqref{G:bound} imply
\[
 \begin{split}
 \E\|\mathcal T_Gu-\mathcal T_Gv\|_{\mathcal X(I)}^2
 &\le C\E\int_s^b\int_Z
        \|G_{r,z}(u(r-)-v(r-))\|_{H^1(\R^3)}^2\nu(\dd z)\dd r\\
 &\le C_VQ\E\int_s^b a(r)^2
              \|u(r-)-v(r-)\|_{H^1(\R^3)}^2\dd r\\
 &\le C_Vq_I\E\|u-v\|_{\mathcal X(I)}^2.
 \end{split}
\]
For the compensator, first estimate pathwise:
\[
 \begin{split}
 \|\mathcal T_Bu-\mathcal T_Bv\|_{\mathcal X(I)}
 &\le C\int_s^b\|B(r)(u(r)-v(r))\|_{H^1(\R^3)}\dd r\\
 &\le C_VQ\int_s^b a(r)^2\dd r\,
               \sup_{s\le r\le b}\|u(r)-v(r)\|_{H^1(\R^3)}.
 \end{split}
\]
Taking its second moment and combining the three inequalities yields
\begin{equation}\label{contraction:constant}
 \|\mathcal Tu-\mathcal Tv\|_{\mathscr B(I)}
 \le\bigl(C\eta^4+C_Vq_I^{1/2}+C_Vq_I\bigr)
                                 \|u-v\|_{\mathscr B(I)}.
\end{equation}
Choose $\eta>0$ so that $C\eta^4\le1/4$, and then choose $\delta>0$
so that $C_V\delta^{1/2}+C_V\delta\le1/4$.
These choices depend only on the profile bounds, the cutoff, and the
constants in Section~\ref{sec:analytic}. They are independent of the
initial datum, the starting time, and the mark truncation, and remain
fixed throughout the construction.

All three nonlinear or stochastic terms vanish when the input is
zero. The free term has $\mathscr B(I)$ norm at most
$C(\E\|f\|_{H^1(\R^3)}^2)^{1/2}$. Thus $\mathcal T$ maps
$\mathscr B(I)$ into itself and is a contraction whenever
$q_I\le\delta$. Its fixed point satisfies
\begin{equation}\label{local:truncatedbound}
 \E\|u\|_{\mathcal X(I)}^2
       \le C\E\|f\|_{H^1(\R^3)}^2.
\end{equation}
The same computation for two fixed points gives
\begin{equation}\label{local:databound}
 \E\|u_f-u_g\|_{\mathcal X(I)}^2
       \le C\E\|f-g\|_{H^1(\R^3)}^2.
\end{equation}
The size of the initial $H^1$ norm enters these bounds but does not
enter the contraction constant.

To restart the construction at a stopping time $\sigma$, define
\begin{equation}\label{local:clockstop}
 \beta(\sigma)
  =(\sigma+1)\wedge
       \inf\{t\ge\sigma:A(t)-A(\sigma)\ge\delta\},
 \qquad \inf\varnothing=\infty.
\end{equation}
The deterministic clock $A$ from \eqref{clock:def} is continuous
and nondecreasing. Consequently $\beta(\sigma)$ is a stopping time,
$\beta(\sigma)>\sigma$ when $\sigma<\infty$, and
\[
 Q\int_\sigma^{\beta(\sigma)}a(r)^2\dd r\le\delta.
\]
Indeed, the event that the clock exit has occurred by $t$ is
$\{\sigma\le t,\ A(t)-A(\sigma)\ge\delta\}$, which belongs to
$\mathcal F_t$; the other endpoint, $\sigma+1$, is also a stopping time.

On $[\sigma,\beta(\sigma)]$, use the expectation of the squared
pathwise $\mathcal X$ norm as in \eqref{local:processnorm}. A process
with initial value $f$ on this stochastic interval may be extended by
zero before $\sigma$ and by its free evolution after the right
endpoint. This convention supplies adapted c\`adl\`ag representatives
on a deterministic time domain. In the integrals use
$\mathbf1_{(\sigma,\beta(\sigma)]}(r)$. This is a predictable
indicator: for each path it is left-continuous as a function of $r$,
and it is adapted. The stochastic Strichartz estimate can therefore
be applied to the extended integrand and then restricted to the
stochastic interval.

The noise-clock bound is applied pathwise before taking expectation:
\[
 \begin{split}
 &Q\E\int_\sigma^{\beta(\sigma)}a(r)^2
             \|u(r-)-v(r-)\|_{H^1(\R^3)}^2\dd r\\
 &\quad\le
 \E\left[
 Q\int_\sigma^{\beta(\sigma)}a(r)^2\dd r\,
 \sup_{\sigma\le r\le\beta(\sigma)}
                     \|u(r)-v(r)\|_{H^1(\R^3)}^2\right]\\
 &\quad\le\delta\,\E\|u-v\|_{\mathcal X([\sigma,\beta(\sigma)])}^2.
 \end{split}
\]
The drift estimate uses the same pathwise bound, and the free
Strichartz estimate is valid after every time translation, with the
same constant. Completeness follows by the subsequence argument
already given, restricted to the stochastic interval. Consequently
\eqref{local:truncatedbound} and \eqref{local:databound} hold also
from a stopping time. If needed, first restrict to
$\{\sigma\le m\}$ and a bounded deterministic terminal time and then
increase $m$.

The fixed point is local with respect to events known at the starting
time: if $E\in\mathcal F_\sigma$ and two initial data coincide on
$E$, their truncated solutions coincide on $E$. To see this, paste
one solution on $E$ and the other on its complement. The pasted
process satisfies the equation with the pasted initial datum,
since $\mathbf1_E\mathbf1_{(\sigma,\beta(\sigma)]}$ is predictable
for the Poisson integral, and the deterministic integrals and cutoff
are evaluated pathwise. Uniqueness of the fixed point gives the assertion.
For an almost surely finite datum $f$, solve with
$\mathbf1_{\{\|f\|_{H^1(\R^3)}\le m\}}f$. Locality makes these
solutions agree on the increasing events
$\{\|f\|_{H^1(\R^3)}\le m\}$. Pasting them constructs the truncated
solution without a moment assumption on $f$.

Starting with $\sigma_0=0$, let $v_k$ be the truncated solution from
$u(\sigma_k)$ on $[\sigma_k,\beta(\sigma_k)]$, and put
\begin{equation}\label{local:criticalstop}
 \sigma_{k+1}
 =\beta(\sigma_k)\wedge
   \inf\{t\in[\sigma_k,\beta(\sigma_k)]:
              \|v_k\|_{Y([\sigma_k,t])}\ge\eta\}.
\end{equation}
The accumulated norm is continuous and adapted, so
$\sigma_{k+1}$ is a stopping time. On this interval the cutoff equals
one, including its right endpoint. Define $u=v_k$ there and use the
post-jump value
$v_k(\sigma_{k+1})$ as the next initial datum. Both exit conditions
involve continuous accumulated integrals. Their values vanish at
$\sigma_k$, so $\sigma_{k+1}>\sigma_k$ almost surely. Additivity of the
integrals and the identity $S(t-r)S(r-s)=S(t-s)$ show that the
concatenated path satisfies the original mild equation. A jump at a
joining time is included in the preceding stochastic integral; the
following integral is open at its left endpoint.

Let $\tau^*=\lim_k\sigma_k$. The concatenation defines an adapted
c\`adl\`ag solution on every compact interval before $\tau^*$.
If $\tau^*\le T<\infty$, at most $A(T)/\delta$ completed steps can
use $\delta$ units of the noise clock, and at most $T$ steps can
have length one. Every other completed step has accumulated
$Y$ norm $\eta$. Since its tenth power is additive over disjoint
time intervals, infinitely many steps imply
\[
 \|u\|_{Y([0,\tau^*))}^{10}
   \ge\sum_{\text{critical exits}}\eta^{10}=\infty.
\]
This proves \eqref{blowup:alternative}. The strict increase also
gives $\sigma_k<\tau^*$ on $\{\tau^*<\infty\}$, so this sequence
announces the lifetime. More quantitatively, whenever a path has
finite $Y([0,T])$ norm, the number of steps needed to cover $[0,T]$
is at most
\begin{equation}\label{local:intervalcount}
 2+T+\frac{A(T)}{\delta}
           +\eta^{-10}\|u\|_{Y([0,T])}^{10}.
\end{equation}
Ties between exit conditions may be assigned to any one of them.

To prove uniqueness in the stated local class, suppose
$u$ and $v$ start from the same datum at $\sigma$. Stop at the first
of $\beta(\sigma)$, the time either accumulated $Y$ norm reaches
$\eta$, and the time either $H^1$ norm exceeds $K$. Localize the
initial event so that its norm is at most $K/2$. Denote this common
stop by $\gamma_K$. Before a possible stopping jump the $H^1$ norms
are at most $K$. Hence the stochastic terms stopped at $\gamma_K$
have second moments bounded by $C_V\delta K^2$. Their values at the
stopping jump are included in this estimate. The deterministic
nonlinear terms have $\mathcal N$ norm at most $C\eta^5$, while the
drift terms have $L^1([\sigma,\gamma_K];H^1(\R^3))$ norm at most
$C_V\delta K$.
The mild equations therefore place both stopped paths in
$\mathscr B([\sigma,\gamma_K])$. The bound concerns their second
moments; their post-jump values at $\gamma_K$ need not be bounded by $K$.

Both paths solve the same truncated equation on this interval.
Applying \eqref{contraction:constant} to their difference gives
\[
 \|u-v\|_{\mathscr B([\sigma,\gamma_K])}
       \le\tfrac12\|u-v\|_{\mathscr B([\sigma,\gamma_K])},
\]
and hence equality of the paths. Increase $K$ and restart from their
common endpoint values. Finitely many such comparisons cover every
compact interval on which both solutions have finite critical norm.
This proves pathwise uniqueness up to the common lifetime.
If an extension of the constructed solution beyond $\tau^*$ existed
in the local class, uniqueness would identify it with $u$ before
$\tau^*$ and contradict \eqref{blowup:alternative}. Thus the
constructed solution is maximal.
\end{proof}

The finite-activity equations admit global solutions on the prescribed
stochastic basis. Constructing these processes before comparing them
with $u$ ensures that the approximation uses the same $u_n$ at every
restart.

\begin{proposition}\label{prop:finite}
For each $n$, the equation with mark integrals restricted to
$Z_n=\{|z|>\eps_n\}$ has a global pathwise solution $u_n$. Between
its jump times it solves
\begin{equation}\label{finite:between}
 \partial_t u_n=i(\Delta u_n-F(u_n))+ia(t)C_nu_n,
 \qquad C_n(x)=\int_{Z_n}A_z(x)\nu(\dd z).
\end{equation}
At a jump $(t,z)$ its value is $J_{t,z}u_n(t-)$.
\end{proposition}

\begin{proof}
The second-moment assumption implies
\[
 \nu(Z_n)\le\eps_n^{-2}Q,\qquad
 \int_{Z_n}|z|\nu(\dd z)\le\eps_n^{-1}Q.
\]
The first inequality gives finitely many $Z_n$ jumps on each
$[0,T]$, almost surely. By the second inequality, each coefficient
$\int_{Z_n}z_j\nu(\dd z)$ is finite. Thus $C_n$ is a well-defined
real linear combination of the profiles, with
\[
 \|C_n\|_{W^{1,\infty}(\R^3)}
    \le C_V\int_{Z_n}|z|\nu(\dd z)
    \le C_V\eps_n^{-1}Q.
\]
In particular,
\[
 \int_0^T\|a(t)C_n\|_{W^{1,\infty}(\R^3)}\dd t
    \le C_V\eps_n^{-1}QT^{1/2}\|a\|_{L^2(0,T)}<\infty.
\]
Lemma~\ref{lem:potential} therefore applies on every interval between
consecutive jumps. These bounds may depend on $n$; uniform bounds
will be obtained from the energy identity in the next section.

To verify the drift in \eqref{finite:between}, write the compensated
Poisson term as its jump sum minus its compensator. The remaining
finite-variation coefficient is
\[
 \int_{Z_n}(R_{t,z}-G_{t,z})\nu(\dd z)
    =ia(t)\int_{Z_n}A_z\nu(\dd z)=ia(t)C_n.
\]
Thus compensation produces the real potential in
\eqref{finite:between}. At a jump,
$u_n(t)=u_n(t-)+G_{t,z}u_n(t-)=J_{t,z}u_n(t-)$.
The weak product rule gives
\[
 \nabla(J_{t,z}f)
    =e^{-ia(t)A_z}\bigl(\nabla f-ia(t)f\nabla A_z\bigr),
\]
and consequently
\[
 \|J_{t,z}f\|_{H^1(\R^3)}
       \le(1+b_1|a(t)||z|)\|f\|_{H^1(\R^3)}.
\]
Each observed jump therefore maps finite $H^1$ data to finite
$H^1$ data, including for unbounded jump amplitudes. Here $a(t)$ is
finite at every time, as stipulated in
Assumption~\ref{ass:model}.

Fix $T$ and list the finitely many jump times on $[0,T]$ in
increasing order.
Solve the deterministic equation up to the first time, apply its
phase map, and continue from the resulting datum. Repeating this
procedure gives a c\`adl\`ag $H^1$ path on $[0,T]$ with finite
$Y([0,T])$ norm. Each deterministic segment has finite norm by
Lemma~\ref{lem:potential}, and there are only finitely many segments.

The local contraction and comparison in Lemma~\ref{lem:potential}
give measurable dependence of the deterministic evolution on its
initial datum. Since the concatenation uses only the jumps already
observed and the evolution from the latest datum, it is adapted.
For observed jumps
$(t_\ell,z_\ell)$, concatenation gives
\[
 \begin{split}
 u_n(t)={}&S(t)u_0-i\int_0^tS(t-r)F(u_n(r))\dd r
       +i\int_0^tS(t-r)a(r)C_nu_n(r)\dd r\\
 &+\sum_{t_\ell\le t}
        S(t-t_\ell)G_{t_\ell,z_\ell}u_n(t_\ell-).
 \end{split}
\]
Substituting the preceding compensation identity recovers the
mark-truncated version of \eqref{main:mild}. Local uniqueness
identifies this construction with its maximal local solution.
The constructions on different finite intervals consequently
agree, which proves global existence for each $n$.
\end{proof}

Let
\[
 Q_n^{\mathrm{tail}}
   =\int_{\{0<|z|\le\eps_n\}}|z|^2\nu(\dd z),\qquad
 h_n(T)=Q_n^{\mathrm{tail}}\int_0^T a(r)^2\dd r.
\]
Dominated convergence gives $Q_n^{\mathrm{tail}}\to0$, and hence
$h_n(T)\to0$ for every finite $T$. These deterministic quantities
measure the omitted noise in the stochastic and compensator estimates.

\begin{lemma}\label{lem:localapprox}
On every interval $I$ of the contraction construction, the truncated
solutions $v_n$ with mark set $Z_n$ converge to the truncated
full-noise solution $v$ in the sense that
$\E\|v_n-v\|_{\mathcal X(I)}^2\to0$ for square-integrable starting
data. The assertion is stable under convergence of the initial data
in $L^2(\Omega;H^1(\R^3))$.

More generally, for each finite $T$ there are increasing stopping
times $\theta_j\le T$ tending to $T\wedge\tau^*$ such that
$\theta_j<\tau^*$ on $\{\tau^*\le T\}$ and
\begin{equation}\label{local:stoppedapprox}
 \|u_n-u\|_{\mathcal X([0,\theta_j])}\longrightarrow0
      \quad\hbox{in probability for each fixed }j.
\end{equation}
Here $u_n$ are the global finite-activity solutions from
Proposition~\ref{prop:finite}. Consequently, the solutions converge
in probability on each compact interval before the maximal
lifetime.
\end{lemma}

\begin{proof}
First let $I=[s,b]\subset[0,T]$ be one contraction interval, whose
endpoints may be stopping times, and let $f_n,f$ be its
$\mathcal F_s$-measurable square-integrable initial data.
Write the two truncated fixed-point equations with the same cutoff
function and threshold. In the difference, all terms involving
$v_n-v$ with the mark set $Z_n$ satisfy the contraction estimate
with constant at most $1/2$, because restricting the mark set only
decreases its second moment. The remaining stochastic term is
\[
 H_n(t)=\int_{(s,t]}\int_{Z\setminus Z_n}
          S(t-r)G_{r,z}v(r-)\widetilde N(\dd r,\dd z).
\]
The second-moment estimate gives, with predictable interval indicators
understood when the endpoints are random,
\[
 \begin{split}
 \E\|H_n\|_{\mathcal X(I)}^2
 &\le C_VQ_n^{\mathrm{tail}}\,
       \E\int_s^b a(r)^2\|v(r-)\|_{H^1(\R^3)}^2\dd r\\
 &\le C_Vh_n(T)\,
       \E\sup_{s\le r\le b}\|v(r)\|_{H^1(\R^3)}^2.
 \end{split}
\]
The other remaining term is the missing compensator
\[
 K_n(t)=\int_s^tS(t-r)
               \int_{Z\setminus Z_n}R_{r,z}v(r)\nu(\dd z)\dd r.
\]
By \eqref{R:bound} and \eqref{str:lone},
\[
 \|K_n\|_{\mathcal X(I)}
     \le C_Vh_n(T)\sup_{s\le r\le b}\|v(r)\|_{H^1(\R^3)}.
\]
The homogeneous estimate handles the initial difference. Absorbing
the contraction term therefore gives
\[
 \begin{split}
 \|v_n-v\|_{\mathscr B(I)}
 \le{}&C\bigl(\E\|f_n-f\|_{H^1(\R^3)}^2\bigr)^{1/2}\\
 &+C_V\bigl(h_n(T)^{1/2}+h_n(T)\bigr)\|v\|_{\mathscr B(I)}.
 \end{split}
\]
Using \eqref{local:truncatedbound} and squaring yields
\begin{equation}\label{local:quantapprox}
 \begin{split}
 \E\|v_n-v\|_{\mathcal X(I)}^2
 \le{}&C\E\|f_n-f\|_{H^1(\R^3)}^2\\
 &+C_V\bigl(h_n(T)+h_n(T)^2\bigr)
                         \E\|f\|_{H^1(\R^3)}^2.
 \end{split}
\end{equation}
The constants in \eqref{local:quantapprox} are independent of $n$
and of the interval endpoints. For a deterministic interval $I$,
$h_n(T)$ can be replaced throughout
by $Q_n^{\mathrm{tail}}\int_I a^2$. If the starting time is unbounded,
one may either localize it or use
$Q_n^{\mathrm{tail}}\int_Ia^2\le
 (Q_n^{\mathrm{tail}}/Q)\delta$ when $Q>0$.
When $Q=0$, the stochastic and compensator terms vanish.

Restarts produce data converging in probability without a uniform
moment bound. To cover this case, let $\Pi_K$ be the metric projection
onto the closed $H^1(\R^3)$ ball of radius $K$:
\[
 \Pi_K f=
 \begin{cases}
 f,&\|f\|_{H^1(\R^3)}\le K,\\
 Kf/\|f\|_{H^1(\R^3)},&\|f\|_{H^1(\R^3)}>K.
 \end{cases}
\]
This projection is continuous and nonexpansive. If $f_n\to f$ in
probability, then $\Pi_Kf_n\to\Pi_Kf$ in probability, and the
difference is bounded in norm by $2K$. For every $\varepsilon>0$,
\[
 \E\|\Pi_Kf_n-\Pi_Kf\|_{H^1(\R^3)}^2
 \le\varepsilon^2+
 4K^2\PP\bigl(\|\Pi_Kf_n-\Pi_Kf\|_{H^1(\R^3)}>\varepsilon\bigr).
\]
Letting $n\to\infty$ and then $\varepsilon\downarrow0$ gives
convergence in $L^2(\Omega;H^1(\R^3))$ for the projected data.
Estimate~\eqref{local:quantapprox} therefore applies.

On the $\mathcal F_s$-measurable event
$\{\|f_n\|_{H^1(\R^3)}\le K,\ \|f\|_{H^1(\R^3)}\le K\}$,
the projected and original truncated solutions agree by locality.
The probability of its complement is small uniformly for sufficiently
large $n$. Indeed,
\[
 \limsup_{n\to\infty}\PP(\|f_n\|_{H^1(\R^3)}>K)
       \le\PP(\|f\|_{H^1(\R^3)}>K/2),
\]
by the triangle inequality and convergence in probability. Increasing
$K$ proves that the original truncated solutions converge in
probability in $\mathcal X(I)$ whenever their initial data do so.

To remove the nonlinear cutoff, we choose comparison intervals using
the limiting path with threshold $\eta/2$, leaving a margin to the
cutoff threshold $\eta$. Fix $T$, put $\theta_0=0$, and, as long as
$\theta_j<T$, set
\begin{equation}\label{local:innerstop}
 \theta_{j+1}
   =T\wedge\beta(\theta_j)\wedge
      \inf\{t\ge\theta_j:
               \|u\|_{Y([\theta_j,t])}\ge\eta/2\}.
\end{equation}
The infimum is taken before the lifetime. The continuous adapted
accumulated norm makes each $\theta_j$ a stopping time. If
$\tau^*\le T$, the
blow-up alternative ensures that this critical exit occurs before
$\tau^*$ whenever the other two exits have not already occurred.
Thus every finite $\theta_j$ is strictly less than $\tau^*$ on this
event. If $\theta_j=T$, keep subsequent times equal to $T$.

These inner intervals exhaust $[0,T\wedge\tau^*)$. In fact, if their
limit were smaller than both $T$ and $\tau^*$, the path would have
finite $Y$ norm on a slightly longer compact interval. Infinitely
many critical exits of size $\eta/2$ would then be impossible,
and finitely many units of physical time and noise clock also rule
out infinitely many other exits. The same reasoning as in
\eqref{local:intervalcount} bounds the number needed before any
compact terminal time by
\[
 2+T+\frac{A(T)}{\delta}
          +(2/\eta)^{10}\|u\|_{Y([0,T])}^{10}
\]
on the event that the right side is defined and finite.

Suppose inductively that $u_n(\theta_j)\to u(\theta_j)$ in
probability. Solve the truncated equations from these data on
$[\theta_j,T\wedge\beta(\theta_j)]$ with the outer threshold $\eta$,
and denote their solutions by $v_n$ and $v$. The preceding argument
gives
\[
 \|v_n-v\|_{\mathcal X([\theta_j,T\wedge\beta(\theta_j)])}
           \longrightarrow0\quad\hbox{in probability}.
\]
First identify these truncated solutions with the original solutions
on stopping-time intervals. Write
$b_j=T\wedge\beta(\theta_j)$ and, for $w=v$ or $v_n$, define
\begin{equation}\label{local:outerstop}
 \rho(w)=\inf\{t\in[\theta_j,b_j]:
                    \|w\|_{Y([\theta_j,t])}\ge\eta\}.
\end{equation}
The accumulated norm is continuous and adapted, so these are stopping
times, with value infinity if the indicated set is empty. Up to
$\rho(v)$, the truncated equation for $v$ has cutoff equal to one.
Pathwise uniqueness on the common stopped interval therefore gives
$v=u$ on $[\theta_j,\theta_{j+1}\wedge\rho(v)]$. If
$\rho(v)\le\theta_{j+1}$, continuity of the accumulated norm would give
\[
 \eta=\|v\|_{Y([\theta_j,\rho(v)])}
      =\|u\|_{Y([\theta_j,\rho(v)])}\le\eta/2.
\]
Consequently $\rho(v)>\theta_{j+1}$, and $v=u$ on the whole inner
interval. For each $n$, the same stopped uniqueness argument, applied
to the mark-truncated equation, gives
\[
 v_n=u_n\quad\text{on }[\theta_j,b_j\wedge\rho(v_n)]
 \quad\text{almost surely}.
\]
These identities are obtained by stopped pathwise uniqueness and
hold outside a null set before any restriction according to the
future approximation error.

Now consider the event that
$\|v_n-v\|_{\mathcal X([\theta_j,b_j])}\le\eta/4$. On this event,
\[
 \|v_n\|_{Y([\theta_j,\theta_{j+1}])}
 \le\|u\|_{Y([\theta_j,\theta_{j+1}])}
        +\|v_n-v\|_{Y([\theta_j,\theta_{j+1}])}
 \le\frac{3\eta}{4}.
\]
In particular, $\rho(v_n)>\theta_{j+1}$ there. The stopped identity
already proved therefore implies $v_n=u_n$ throughout the inner
interval on this event. This restriction uses a pathwise identity;
it does not insert the future error event into a stochastic integrand.
For any $\varepsilon>0$ we have obtained
\[
 \begin{split}
 &\PP\bigl(
   \|u_n-u\|_{\mathcal X([\theta_j,\theta_{j+1}])}
                                      >\varepsilon\bigr)\\
 &\quad\le
 \PP\bigl(
  \|v_n-v\|_{\mathcal X([\theta_j,T\wedge\beta(\theta_j)])}
                            >\min\{\varepsilon,\eta/4\}\bigr)
       \longrightarrow0.
 \end{split}
\]
The $\mathcal X$ supremum includes $\theta_{j+1}$, so the same
estimate gives convergence of the endpoint data, including a possible
jump, for the next restart.

At $j=0$ the data agree. Induction consequently proves convergence
on every fixed finite collection of inner intervals. The supremum
over their union is the maximum of the separate suprema, and the
tenth power of its $Y$ norm is the sum of the separate tenth powers.
Consequently, the $\mathcal X$ norm on the union is bounded by the
sum of the $\mathcal X$ norms on the individual intervals. This gives
\eqref{local:stoppedapprox}.

To state the local convergence without a random norm domain, fix
$t\le T$. On $\{t<\tau^*\}$, finitely many inner
intervals cover $[0,t]$ almost surely. For each fixed $j$,
\[
 \begin{split}
 &\PP\bigl(
   \|u_n-u\|_{\mathcal X([0,t])}>\varepsilon,\ t<\tau^*\bigr)\\
 &\quad\le
 \PP(\theta_j<t,\ t<\tau^*)+
 \PP\bigl(\|u_n-u\|_{\mathcal X([0,\theta_j])}>\varepsilon\bigr).
 \end{split}
\]
Let $n\to\infty$ and then $j\to\infty$. The second term vanishes
by \eqref{local:stoppedapprox}, and the first vanishes by the
exhaustion just established. This proves the asserted local
convergence in probability.
\end{proof}

\section{Energy bounds and global continuation}\label{sec:energy}

The energy argument separates the martingale part of each jump from
its nonnegative quadratic correction. The former is linear in the
mark and is estimated in $L^2(\Omega)$; the latter is estimated in
$L^1(\Omega)$ by compensation. Both estimates therefore require only
the second moment of the L\'evy measure. We first obtain bounds uniform
in the finite-activity truncation, transfer them to the maximal local
solution, and prove global continuation. The exact energy identity
for the limiting equation is established after this construction.
Throughout the section, $u_0$ is deterministic and
$M_0=M(u_0)=\|u_0\|_{L^2(\R^3)}^2$. The constants in the maximal
energy estimate are independent of the truncation of the L\'evy measure.
The factor $b_1$ records the dependence on the spatial gradients of
the phase profiles; their spatially constant parts leave the energy
unchanged.

\begin{lemma}\label{lem:jumpenergy}
For $f\in H^1(\R^3)$,
\begin{equation}\label{jump:energy}
 M(J_{t,z}f)=M(f),\qquad
 E(J_{t,z}f)-E(f)=\Lambda_{t,z}(f)+D_{t,z}(f).
\end{equation}
Moreover,
\begin{align}
 |\Lambda_{t,z}(f)|^2
  &\le2a(t)^2|z|^2b_1^2 M(f)E(f),\label{Lambda:bound}\\
 0\le D_{t,z}(f)
  &\le\tfrac12a(t)^2|z|^2b_1^2M(f).\label{D:bound}
\end{align}
\end{lemma}

\begin{proof}
Let $\vartheta=a(t)A_z$. Since $A_z$ is real and Lipschitz, the
Sobolev product and chain rules give
\begin{equation}\label{jump:gradient}
 \nabla(e^{-i\vartheta}f)
       =e^{-i\vartheta}(\nabla f-if\nabla\vartheta)
       \quad\text{almost everywhere on }\R^3.
\end{equation}
The formula also follows by smooth approximation in $H^1(\R^3)$,
using boundedness of the multiplier and its first derivative. Thus
only the assumed first spatial derivatives of $A_z$ are needed.
The phase has modulus one, so
\[
 \|J_{t,z}f\|_{L^2(\R^3)}=\|f\|_{L^2(\R^3)},\qquad
 \|J_{t,z}f\|_{L^6(\R^3)}=\|f\|_{L^6(\R^3)}.
\]
The mass and the nonlinear part of the energy are therefore unchanged.
The cross term in the square of \eqref{jump:gradient} is determined by
\[
 \operatorname{Re}\bigl(\nabla f\cdot
               \overline{-ia(t)f\nabla A_z}\bigr)
 =a(t)\nabla A_z\cdot\operatorname{Re}(i\overline f\nabla f)
 =-a(t)\nabla A_z\cdot j(f).
\]
Consequently,
\[
 \begin{split}
 |\nabla(J_{t,z}f)|^2
 &=|\nabla f-ia(t)f\nabla A_z|^2\\
 &=|\nabla f|^2
   -2a(t)\nabla A_z\cdot\operatorname{Im}(\overline f\nabla f)
   +a(t)^2|f|^2|\nabla A_z|^2.
 \end{split}
\]
All terms in this identity are integrable: the mixed term is in
$L^1(\R^3)$ by Cauchy--Schwarz, and the last term is integrable by
boundedness of $\nabla A_z$. Integrating and dividing by two proves
\eqref{jump:energy}, with exactly the sign in
\eqref{energy:increments}. Cauchy--Schwarz in the coefficient index
first gives
\[
 |\nabla A_z(x)|
 \le |z|\Bigl(\sum_{j=1}^m|\nabla V_j(x)|^2\Bigr)^{1/2}
 \le b_1|z|.
\]
Cauchy--Schwarz in the spatial variable then gives
\[
 |\Lambda_{t,z}(f)|
       \le|a(t)||z|b_1\|f\|_{L^2(\R^3)}
                                  \|\nabla f\|_{L^2(\R^3)}.
\]
Squaring and using
$\|\nabla f\|_{L^2(\R^3)}^2\le2E(f)$ proves
\eqref{Lambda:bound}. Similarly,
\[
 D_{t,z}(f)=\frac{a(t)^2}{2}
                \int_{\R^3}|f(x)|^2|\nabla A_z(x)|^2\dd x
 \le\frac{a(t)^2}{2}b_1^2|z|^2M(f),
\]
which proves \eqref{D:bound}. The identity is exact for every mark
$z$, and its energy correction has no terms of order higher than two.
\end{proof}

\begin{proposition}\label{prop:finiteenergy}
The finite-activity solution $u_n$ conserves mass and satisfies
\begin{equation}\label{finite:balance}
 \begin{split}
 E(u_n(t))=E(u_0)
 &+\int_0^t\int_{Z_n}\Lambda_{s,z}(u_n(s-))
                                      \widetilde N(\dd s,\dd z)\\
 &+\int_0^t\int_{Z_n}D_{s,z}(u_n(s-))N(\dd s,\dd z).
 \end{split}
\end{equation}
For every finite $T$, uniformly in $n$,
\begin{equation}\label{finite:maxbound}
 \E\sup_{0\le t\le T}E(u_n(t))
       \le C\bigl(E(u_0)+b_1^2M(u_0)A(T)\bigr).
\end{equation}
\end{proposition}

\begin{proof}
For a fixed $n$, the number of $Z_n$ jumps before a fixed finite time
is almost surely finite. The first-moment bound on $Z_n$ makes
$C_n=\int_{Z_n}A_z\nu(\dd z)$ a well-defined real
$W^{1,\infty}(\R^3)$ function. Expanding the compensated equation
between jumps, its drift from the noise is
\[
 \int_{Z_n}(R_{t,z}-G_{t,z})u_n\,\nu(\dd z)
 =ia(t)\int_{Z_n}A_z u_n\,\nu(\dd z)
 =ia(t)C_nu_n.
\]
This gives the potential term in \eqref{finite:between} with its
stated sign. Since $a\in L^1(0,T)$, the potential lemma
applies on every interval between consecutive jumps. It preserves
mass there, and Lemma~\ref{lem:jumpenergy} preserves mass at every
jump. Hence $M(u_n(t))=M_0$ for all $t$ on an event of full probability.

On an interval without a jump, \eqref{potential:energy} with
$P=aC_n$ gives
\[
 \frac{\mathrm d}{\mathrm dt}E(u_n(t))
     =a(t)\int_{\R^3}\nabla C_n\cdot j(u_n(t))\dd x
     =-\int_{Z_n}\Lambda_{t,z}(u_n(t))\nu(\dd z)
\]
for almost every $t$. The interchange of the $z$ and $x$ integrals is
justified by $\int_{Z_n}|z|\nu(\dd z)<\infty$ and the spatial
Cauchy--Schwarz bound used above. At a jump with mark $z$, the energy
increment is $\Lambda_{t,z}(u_n(t-))+D_{t,z}(u_n(t-))$. Adding the
absolutely continuous increments and the finite sum of jump increments
therefore yields
\[
 \begin{split}
 E(u_n(t))-E(u_0)
 ={}&-\int_0^t\int_{Z_n}\Lambda_{s,z}(u_n(s))\nu(\dd z)\dd s\\
 &+\int_0^t\int_{Z_n}
       (\Lambda_{s,z}(u_n(s-))+D_{s,z}(u_n(s-)))N(\dd s,\dd z).
 \end{split}
\]
Changing to left limits in the time integral does not change its value.
Combining its $\Lambda$ term with the Poisson sum gives
\eqref{finite:balance}. The compensating drift cancels exactly at
each truncation level, so the energy estimate does not require a
uniform bound on $C_n$.

Fix $R>E(u_0)$ and define
\[
 \rho_R=\inf\{t\ge0:E(u_n(t))\ge R\},
\]
with the infimum of the empty set equal to infinity. Write
\[
 \begin{split}
 \mathcal L_R(t)&=\int_0^{t\wedge\rho_R}\int_{Z_n}
       \Lambda_{s,z}(u_n(s-))\widetilde N(\dd s,\dd z),\\
 \mathcal D_R(t)&=\int_0^{t\wedge\rho_R}\int_{Z_n}
       D_{s,z}(u_n(s-))N(\dd s,\dd z).
 \end{split}
\]
These integrals include a possible jump at $\rho_R$ and use predictable
left-limit values. In particular, $E(u_n(s-))\le R$ for
$0<s\le\rho_R$, even when the stopping jump overshoots $R$. Thus
\begin{equation}\label{finite:stopped-quadratic}
 \E|\mathcal L_R(T)|^2
 =\E\int_0^{T\wedge\rho_R}\int_{Z_n}
      |\Lambda_{s,z}(u_n(s-))|^2\nu(\dd z)\dd s
 \le2b_1^2M_0R A(T)<\infty.
\end{equation}
The nonnegative integral, by the compensation formula and
\eqref{D:bound}, satisfies
\begin{equation}\label{finite:positive-expectation}
 \E\mathcal D_R(T)
 =\E\int_0^{T\wedge\rho_R}\int_{Z_n}
                   D_{s,z}(u_n(s-))\nu(\dd z)\dd s
 \le\frac12b_1^2M_0 A(T).
\end{equation}
Both bounds depend only on the pre-jump mass and energy and remain
valid for an arbitrarily large stopping jump.

Set $H_R=\E\sup_{0\le t\le T}E(u_n(t\wedge\rho_R))$.
The stopped balance and nonnegativity of $\mathcal D_R$ imply
\[
 \sup_{0\le t\le T}E(u_n(t\wedge\rho_R))
 \le E(u_0)+\sup_{0\le t\le T}|\mathcal L_R(t)|+\mathcal D_R(T).
\]
By \eqref{finite:stopped-quadratic}, Doob's inequality and
\eqref{finite:positive-expectation}, the right-hand side is integrable.
Thus $H_R<\infty$, so the subsequent absorption is justified. Replacing
the bound $R$ in \eqref{finite:stopped-quadratic} by the stopped energy
supremum gives
\[
 \E|\mathcal L_R(T)|^2
 \le2b_1^2M_0Q\int_0^T a(s)^2
       \E\bigl[\mathbf1_{\{s\le\rho_R\}}E(u_n(s-))\bigr]\dd s
 \le2b_1^2M_0 A(T)H_R.
\]
Doob's $L^2$ inequality followed by Cauchy--Schwarz consequently gives
\begin{equation}\label{finite:doob}
 \begin{split}
 \E\sup_{0\le t\le T}|\mathcal L_R(t)|
 &\le\left(\E\sup_{0\le t\le T}|\mathcal L_R(t)|^2\right)^{1/2}\\
 &\le2\bigl(\E|\mathcal L_R(T)|^2\bigr)^{1/2}
 \le2\sqrt2\,b_1\sqrt{M_0A(T)H_R}.
 \end{split}
\end{equation}
Substituting this into the stopped balance yields
\[
 H_R\le E(u_0)+2\sqrt2\,b_1\sqrt{M_0A(T)H_R}
                         +\frac12b_1^2M_0A(T).
\]
The elementary inequality
\[
 2\sqrt2\,b_1\sqrt{M_0A(T)H_R}
 \le\frac12H_R+4b_1^2M_0A(T)
\]
therefore gives the explicit uniform estimate
\begin{equation}\label{finite:uniform-stopped}
 H_R\le2E(u_0)+9b_1^2M_0A(T).
\end{equation}
The finite-activity path is bounded in $H^1(\R^3)$ on a finite interval,
since it contains only finitely many phase maps and continuous
deterministic solution pieces. Thus $\rho_R\to\infty$ as $R\to\infty$
on each such path. For sufficiently large $R$, the stopped and
unstopped paths coincide on $[0,T]$. Fatou's lemma applied to
\eqref{finite:uniform-stopped} proves \eqref{finite:maxbound}.
Here the square-integrability estimate is applied only to $\Lambda$,
which is linear in $z$. The quadratic term $D$ is retained as a
nonnegative Poisson integral and controlled through
\eqref{finite:positive-expectation}; no fourth jump moment is used.
\end{proof}

\begin{lemma}\label{lem:prelife}
The maximal solution of Proposition~\ref{prop:local} conserves mass
before its lifetime and satisfies
\begin{equation}\label{prelife:energy}
 \E\sup_{0\le t<T\wedge\tau^*} E(u(t))
       \le C\bigl(E(u_0)+b_1^2M(u_0)A(T)\bigr)
\end{equation}
for every finite $T$.
\end{lemma}

\begin{proof}
Fix $T<\infty$ and use the inner stopping times $\theta_j\le T$
defined in \eqref{local:innerstop}. Lemma~\ref{lem:localapprox}
shows that these intervals exhaust $[0,T\wedge\tau^*)$, include
$T$ after finitely many steps when $T<\tau^*$, and have endpoints
strictly before $\tau^*$ when $\tau^*\le T$. For each fixed $j$,
\eqref{local:stoppedapprox} gives
\begin{equation}\label{prelife:inner-convergence}
 \|u_n-u\|_{\mathcal X([0,\theta_j])}
       \longrightarrow0\qquad\text{in probability}.
\end{equation}
This convergence concerns the original global finite-activity solutions
$u_n$, including their post-jump values at $\theta_j$. It follows from
the local construction and projection of the data at stopping times,
independently of global existence for the limiting solution.

The energy is locally Lipschitz on $H^1(\R^3)$. For $f$ and $g$ in
the closed $H^1(\R^3)$ ball of radius $K$, we have
\begin{align}
 |E(f)-E(g)|
 &\le\frac12\bigl(\|\nabla f\|_{L^2(\R^3)}
                  +\|\nabla g\|_{L^2(\R^3)}\bigr)
                    \|\nabla(f-g)\|_{L^2(\R^3)}\notag\\
 &\quad+C\bigl(\|f\|_{L^6(\R^3)}
                  +\|g\|_{L^6(\R^3)}\bigr)^5
                    \|f-g\|_{L^6(\R^3)}\notag\\
 &\le C(K+K^5)\|f-g\|_{H^1(\R^3)}.\label{energy:local-lipschitz}
\end{align}
The mass satisfies the analogous quadratic estimate. Fix $j$. To
retain the limit inferior of the full sequence when applying Fatou's
lemma, first choose a subsequence along which
$\E\sup_{0\le t\le T}E(u_n(t))$ converges to the limit inferior of
this sequence. From that subsequence, convergence in probability in
\eqref{prelife:inner-convergence} permits a further subsequence
converging almost surely. Each limiting path is bounded in $H^1(\R^3)$ on
$[0,\theta_j]$, by its c\`adl\`ag regularity on this compact
interval. Uniform convergence makes the approximating paths bounded
there as well, for all sufficiently large indices of the subsequence.
Applying \eqref{energy:local-lipschitz} with this finite, possibly
path-dependent bound proves uniform convergence of their energies.
The same argument proves uniform convergence of their masses.
Since the approximating mass equals $M_0$ at all times, the limiting
mass is $M_0$ throughout this interval.

Fatou's lemma and Proposition~\ref{prop:finiteenergy} now give
\begin{align}
 \E\sup_{0\le t\le\theta_j}E(u(t))
 &\le\liminf_{n\to\infty}
          \E\sup_{0\le t\le T}E(u_n(t))\notag\\
 &\le C\bigl(E(u_0)+b_1^2M_0A(T)\bigr).\label{prelife:inner-fatou}
\end{align}
The second subsequence preserves the limit of the expectations, which
justifies the limit inferior over the full sequence in
\eqref{prelife:inner-fatou}. The subsequences may depend on $j$.
Intersecting the resulting countably many full-probability events
gives mass conservation on the union of the inner intervals.

The energy suprema in \eqref{prelife:inner-fatou} increase with $j$.
The exhaustion therefore gives the supremum before $T\wedge\tau^*$,
and also includes $u(T)$ when $T<\tau^*$. Monotone convergence proves
\eqref{prelife:energy}. The order of these limits uses only local
approximation and the finite-activity estimate.
\end{proof}

\begin{proposition}\label{prop:continuation}
The maximal lifetime $\tau^*$ is infinite almost surely.
\end{proposition}

\begin{proof}
Fix $T<\infty$. Lemma~\ref{lem:prelife} implies
\begin{equation}\label{prelife:Hone}
 \E\sup_{0\le t<T\wedge\tau^*}\|u(t)\|_{H^1(\R^3)}^2<\infty.
\end{equation}
We extend the stochastic integrand to $[0,T]$ while preserving
predictability. On each interval
$(\sigma_k,\sigma_{k+1}]$ in the announcing construction, use the
predictable left-limit integrand of that local solution. Paste these
integrands, include the first interval $(0,\sigma_1]$, and set the
result to zero outside their union. Since $\sigma_k<\tau^*$ on a finite
lifetime, this gives the predictable process
\[
 \widehat G(s,z)=\mathbf1_{\{s<\tau^*\}}G_{s,z}u(s-),
\]
with value zero at $s=\tau^*$. This definition by predictable intervals
also covers the event of infinite lifetime. Define $\widehat B(s)$
in the same way from $B(s)u(s)$; changing to left limits does not alter
its time integral.
For stopping times $\sigma\le\tau$, the process
$\mathbf1_{\{\sigma<s\le\tau\}}$ is predictable: its paths are
left-continuous and it is adapted. The countable pasted sum therefore
preserves predictability. At a construction endpoint the left-limit
value is assigned by the interval ending there; the next interval is
open on the left. At a finite lifetime the increasing union of the
construction intervals excludes the lifetime itself. Thus this
definition does not require a value of the solution at $\tau^*$.

By \eqref{G:bound} and \eqref{prelife:Hone},
\[
 \E\int_0^T\int_Z\|\widehat G(s,z)\|_{H^1(\R^3)}^2\nu(\dd z)\dd s
 \le C_V A(T)\E\sup_{0\le s<T\wedge\tau^*}
                                      \|u(s)\|_{H^1(\R^3)}^2<\infty.
\]
Consequently,
\begin{equation}\label{prelife:martingale}
 P(t)=\int_0^t\int_Z S(-s)\widehat G(s,z)\widetilde N(\dd s,\dd z)
\end{equation}
is a c\`adl\`ag $H^1(\R^3)$-valued martingale on $[0,T]$.
Lemma~\ref{lem:poisson} also gives
\begin{equation}\label{prelife:W}
W(t):=S(t)P(t)\in Y([0,T])\quad\text{almost surely}.
\end{equation}
In particular, $P$ has a left limit in $H^1(\R^3)$ at every positive
time in $[0,T]$, including a possible random lifetime. The version on
the fixed deterministic interval provides this limit simultaneously
at every time. It is the input needed to control the restarted
convolution near the possible lifetime.
Similarly, the process
\[
 C(t)=\int_0^t S(-s)\widehat B(s)\dd s
\]
is continuous with finite variation in $H^1(\R^3)$, since
\begin{equation}\label{prelife:drift}
 \int_0^T\|\widehat B(s)\|_{H^1(\R^3)}\dd s
    \le C_V A(T)\sup_{0\le s<T\wedge\tau^*}
                                    \|u(s)\|_{H^1(\R^3)}<\infty
\end{equation}
almost surely.

Fix a sample path satisfying all these properties and suppose that
$\tau^*\le T$. For $s<t<\tau^*$, the mild equation can be restarted
pathwise as
\begin{equation}\label{restart:prelife}
 u(t)=S(t-s)u(s)-i\int_s^t S(t-r)F(u(r))\dd r+\Xi_s(t),
\end{equation}
where
\[
 \Xi_s(t)=S(t)(P(t)-P(s))+S(t)(C(t)-C(s)).
\]
We claim that
\begin{equation}\label{prelife:tailsmall}
 \|\Xi_s\|_{\mathcal X([s,\tau^*))}\longrightarrow0
                    \quad\text{as }s\uparrow\tau^*.
\end{equation}
For the first term, the existence of $P(\tau^* -)$ gives
\[
 \sup_{s\le t<\tau^*}\|P(t)-P(s)\|_{H^1(\R^3)}\longrightarrow0.
\]
For any $\eps>0$, choose $s_0<\tau^*$ such that
\[
 \|P(r)-P(\tau^*-)\|_{H^1(\R^3)}<\eps
 \qquad(s_0\le r<\tau^*).
\]
The supremum above is then at most $2\eps$
for $s\ge s_0$. Unitarity of $S(t)$ gives the same bound for
$S(t)(P(t)-P(s))$.
For its spacetime norm, add and subtract $S(t)P(\tau^* -)$ to get
\[
 \begin{split}
 \|S(t)(P(t)-P(s))\|_{Y([s,\tau^*))}
 \le{}&\|W\|_{Y([s,\tau^*))}
       +\|S(t)P(\tau^* -)\|_{Y([s,\tau^*))}\\
     &+C\|P(s)-P(\tau^* -)\|_{H^1(\R^3)}.
 \end{split}
\]
The first two terms tend to zero by absolute continuity of their
tenth-power time integrals, using \eqref{prelife:W} and the homogeneous
Strichartz estimate. The last tends to zero by the left limit. For the
finite-variation term, \eqref{str:lone} gives
\[
 \|S(t)(C(t)-C(s))\|_{\mathcal X([s,\tau^*))}
       \le C\int_s^{\tau^*}\|\widehat B(r)\|_{H^1(\R^3)}\dd r
          \longrightarrow0.
\]
This proves \eqref{prelife:tailsmall}. The use of the fixed vector
$P(\tau^*-)$ is essential: its free evolution has a vanishing
spacetime norm on the shrinking interval, while the remaining free
evolution is small by $H^1(\R^3)$ convergence.

Set $K=1+\sup_{0\le t<\tau^*}\|u(t)\|_{H^1(\R^3)}$, which is finite
on the fixed path. Choose $s<\tau^*$ sufficiently close to $\tau^*$ so
that the norm in \eqref{prelife:tailsmall} is at most $\delta_*(K)$.
Let $v_s$ be the unforced solution with $v_s(s)=u(s)$. By
Lemma~\ref{lem:detglobal}, its global critical norm is at most
$C_0(K)$. For each finite $b$ with $s<b<\tau^*$, the stochastic
trajectory already belongs to $\mathcal X([s,b])$ by local existence.
The external term $\Xi_s$ vanishes at $s$ and has norm at most
$\delta_*(K)$ on this interval, because restriction cannot increase
the norm. All the hypotheses of Lemma~\ref{lem:stability} thus hold
on $[s,b]$, with the same constants for every such $b$. It gives
\begin{equation}\label{continuation:uniform-critical}
 \|u\|_{Y([s,b])}
 \le C_0(K)+C_*(K)\|\Xi_s\|_{\mathcal X([s,\tau^*))}
 \le C_0(K)+C_*(K)\delta_*(K).
\end{equation}
Increase $b$ to $\tau^*$ through a sequence and use monotone
convergence of the tenth-power spacetime integrals. The bound
\eqref{continuation:uniform-critical} proves
$\|u\|_{Y([s,\tau^*))}<\infty$. The norm on $[0,s]$ is finite by
local existence. Hence the full critical norm before $\tau^*$ is
finite, contradicting \eqref{blowup:alternative}. The comparison is
applied only on compact intervals inside the known lifetime, with
the endpoint reached through the uniform bound.

The time $s$ is chosen after fixing the sample path and is used only
in this deterministic comparison. The stochastic integral was defined
on $[0,T]$ in \eqref{prelife:martingale}. We have proved
$\PP(\tau^*\le T)=0$. Increasing integer $T$ proves the proposition.
\end{proof}

\begin{proof}[Completion of Theorem~\ref{thm:global}]
Global existence and pathwise uniqueness follow from
Propositions~\ref{prop:local} and \ref{prop:continuation}. The mass
identity and \eqref{energy:bound:intro} follow from
Lemma~\ref{lem:prelife}, first on $[0,T)$ and then on $[0,T]$ by
using a larger deterministic time. The solution is constructed as an
adapted fixed-point limit and subsequent concatenation on the original
stochastic basis, so it is probabilistically strong.

For convergence on $[0,T]$, use the same fixed-$T$ inner endpoints
$\theta_j\le T$ from \eqref{local:innerstop}. Since $\tau^*=\infty$
almost surely, the finite-interval exhaustion in
Lemma~\ref{lem:localapprox} shows that $\theta_j=T$ for all
sufficiently large $j$ on each path. For every fixed integer $J$,
\eqref{local:stoppedapprox} gives
\begin{equation}\label{global:fixed-restart-error}
 \|u_n-u\|_{\mathcal X([0,\theta_J])}
       \longrightarrow0\quad\text{in probability}.
\end{equation}
Although the number of restarts needed to reach $T$ is random, the
local approximation is used with a fixed $J$ before this index tends
to infinity.

For $\eps>0$, the event that the error on $[0,T]$ exceeds $\eps$
is contained in the union of $\{\theta_J<T\}$ and the event that
the error in \eqref{global:fixed-restart-error} exceeds $\eps$.
Consequently,
\begin{align}
 &\limsup_{n\to\infty}
 \PP\bigl(\|u_n-u\|_{\mathcal X([0,T])}>\eps\bigr)\notag\\
 &\qquad\le\PP(\theta_J<T).
 \label{global:probability-exhaustion}
\end{align}
Let $J\to\infty$. The right-hand side tends to zero, which proves
\eqref{approx:intro}. The event $\{\theta_J<T\}$ enters only this
probability comparison, after the local stochastic estimates have
been established.

If deterministic initial data $u_{0,n}$ converge to $u_0$ in
$H^1(\R^3)$ and the mark measure is unchanged, the local difference
estimate has an initial term
$C\|u_{0,n}-u_0\|_{H^1(\R^3)}$ and no missing-mark term.
At later restarts the data converge in probability, and the projection
argument of Lemma~\ref{lem:localapprox} gives the required local
comparison. Thus \eqref{global:fixed-restart-error} holds for these
solutions as well. Applying \eqref{global:probability-exhaustion}
proves the stated continuous dependence in probability.
\end{proof}

\begin{proof}[Proof of Proposition~\ref{prop:balance}]
We first construct the limiting integrals from the established energy
bound, and then pass to the limit in the finite-activity identity. By
\eqref{Lambda:bound}, mass conservation and
\eqref{energy:bound:intro},
\begin{equation}\label{balance:limiting-integrability}
 \E\int_0^T\int_Z|\Lambda_{s,z}(u(s-))|^2\nu(\dd z)\dd s
 \le2b_1^2M_0A(T)\E\sup_{0\le s\le T}E(u(s))<\infty.
\end{equation}
Thus its compensated integral is a square-integrable real martingale.
Also, the compensation formula and \eqref{D:bound} give
\begin{equation}\label{balance:positive-integrability}
 \E\int_0^T\int_ZD_{s,z}(u(s-))N(\dd s,\dd z)
 \le\frac12b_1^2M_0A(T)<\infty.
\end{equation}
The latter integral has nonnegative, finite-variation c\`adl\`ag
paths. Denote these two processes by $\mathcal L(t)$ and
$\mathcal D(t)$, respectively. Denote their finite-activity versions
with $u_n$ and $Z_n$ by $\mathcal L_n(t)$ and $\mathcal D_n(t)$.

The following difference estimates specify the integrability needed
to pass to the limit:
\begin{equation}\label{Lambda:diff}
 \begin{split}
 |\Lambda_{t,z}(f)-\Lambda_{t,z}(g)|
 &\le b_1|a(t)||z|
       (\|f\|_{H^1(\R^3)}+\|g\|_{H^1(\R^3)})
                                \|f-g\|_{H^1(\R^3)},\\
 |D_{t,z}(f)-D_{t,z}(g)|
 &\le \tfrac12 b_1^2a(t)^2|z|^2
       (\|f\|_{L^2(\R^3)}+\|g\|_{L^2(\R^3)})
                                \|f-g\|_{L^2(\R^3)}.
 \end{split}
\end{equation}
The first inequality follows by writing the current difference as
\[
 \overline f\nabla f-\overline g\nabla g
       =\overline{(f-g)}\nabla f+\overline g\nabla(f-g).
\]
Indeed, the $L^1(\R^3)$ norm of this expression is at most
\[
 \|f-g\|_{L^2(\R^3)}\|\nabla f\|_{L^2(\R^3)}
       +\|g\|_{L^2(\R^3)}\|\nabla(f-g)\|_{L^2(\R^3)}.
\]
Multiplying by $|a(t)|\|\nabla A_z\|_{L^\infty(\R^3)}$ gives
the first inequality. For the second, use
\[
 \begin{split}
 \bigl|\|f\nabla A_z\|_{L^2(\R^3)}^2
       -\|g\nabla A_z\|_{L^2(\R^3)}^2\bigr|\\
 \le\|\nabla A_z\|_{L^\infty(\R^3)}^2
       (\|f\|_{L^2(\R^3)}+\|g\|_{L^2(\R^3)})
                          \|f-g\|_{L^2(\R^3)}.
 \end{split}
\]

Fix $T<\infty$ and $K>\|u_0\|_{H^1(\R^3)}$. Let
\begin{equation}\label{balance:common-stop}
 \rho_{K,n}=\inf\{t\ge0:
       \|u_n(t)\|_{H^1(\R^3)}\vee\|u(t)\|_{H^1(\R^3)}\ge K\}.
\end{equation}
For $0<s\le\rho_{K,n}$, the two left-limit norms are at most $K$.
The values at the exit itself need not satisfy that bound. Define
\[
 d_{K,n}=\sup_{0<s\le T\wedge\rho_{K,n}}
       \|u_n(s-)-u(s-)\|_{H^1(\R^3)}.
\]
Then $d_{K,n}\le2K$ and
$d_{K,n}\le\sup_{0\le s\le T}\|u_n(s)-u(s)\|_{H^1(\R^3)}$.
Theorem~\ref{thm:global} therefore gives convergence of $d_{K,n}$
to zero in probability. Its boundedness upgrades this to
\begin{equation}\label{balance:bounded-difference}
 \E d_{K,n}^2\longrightarrow0,\qquad
 \E d_{K,n}\longrightarrow0.
\end{equation}
Indeed, for any $h>0$,
$\E d_{K,n}^2\le h^2+4K^2\PP(d_{K,n}>h)$; first let
$n\to\infty$ and then $h\downarrow0$.

The stopped martingale difference has integrand
\begin{equation}\label{balance:martingale-split}
 \begin{split}
 &\mathbf1_{Z_n}(z)\Lambda_{s,z}(u_n(s-))
                          -\Lambda_{s,z}(u(s-))\\
 &\quad=\mathbf1_{Z_n}(z)
       \bigl(\Lambda_{s,z}(u_n(s-))-\Lambda_{s,z}(u(s-))\bigr)
       -\mathbf1_{Z\setminus Z_n}(z)\Lambda_{s,z}(u(s-)).
 \end{split}
\end{equation}
The whole expression is multiplied by the predictable indicator
$\mathbf1_{\{s\le\rho_{K,n}\}}$. By \eqref{Lambda:diff},
\begin{align}
 &\E\int_0^{T\wedge\rho_{K,n}}\int_{Z_n}
    |\Lambda_{s,z}(u_n(s-))-\Lambda_{s,z}(u(s-))|^2
                                           \nu(\dd z)\dd s\notag\\
 &\qquad\le4b_1^2K^2A(T)\E d_{K,n}^2.\label{balance:Lambda-Ltwo}
\end{align}
For the omitted marks, mass conservation and the pre-jump gradient
bound give
\begin{align}
 &\E\int_0^{T\wedge\rho_{K,n}}\int_{Z\setminus Z_n}
               |\Lambda_{s,z}(u(s-))|^2\nu(\dd z)\dd s\notag\\
 &\qquad\le b_1^2M_0K^2 Q_n^{\mathrm{tail}}
                                    \int_0^T a(s)^2\dd s.
 \label{balance:Lambda-missing}
\end{align}
Doob's inequality and the real martingale isometry applied to
\eqref{balance:martingale-split} show that
\begin{align}
 &\E\sup_{0\le t\le T}
   |\mathcal L_n(t\wedge\rho_{K,n})
                   -\mathcal L(t\wedge\rho_{K,n})|^2\notag\\
 &\qquad\le Cb_1^2K^2
       \left(A(T)\E d_{K,n}^2
          +M_0Q_n^{\mathrm{tail}}\int_0^T a(s)^2\dd s\right)
       \longrightarrow0.\label{balance:martingale-convergence}
\end{align}

For the positive Poisson integrals, use total variation of the
difference and then the compensation formula. Since both masses
equal $M_0$, \eqref{Lambda:diff} gives
\begin{align}
 &\E\int_0^{T\wedge\rho_{K,n}}\int_{Z_n}
       |D_{s,z}(u_n(s-))-D_{s,z}(u(s-))|N(\dd s,\dd z)\notag\\
 &\qquad=\E\int_0^{T\wedge\rho_{K,n}}\int_{Z_n}
       |D_{s,z}(u_n(s-))-D_{s,z}(u(s-))|\nu(\dd z)\dd s\notag\\
 &\qquad\le b_1^2\sqrt{M_0}\,A(T)\E d_{K,n}.
 \label{balance:D-Lone}
\end{align}
The analogous missing-mark contribution is at most
\begin{equation}\label{balance:D-missing}
 \frac12b_1^2M_0 Q_n^{\mathrm{tail}}\int_0^T a(s)^2\dd s.
\end{equation}
It follows that
\begin{equation}\label{balance:positive-convergence}
 \E\sup_{0\le t\le T}
   |\mathcal D_n(t\wedge\rho_{K,n})
                    -\mathcal D(t\wedge\rho_{K,n})|
       \longrightarrow0.
\end{equation}
The limit passage preserves the same moment structure as the energy
estimate: \eqref{balance:martingale-convergence} is an $L^2(\Omega)$
bound for the term linear in $z$, and
\eqref{balance:positive-convergence} is an $L^1(\Omega)$ bound for
the term quadratic in $z$. Both include the common stopping jump.
Their predictable integrands involve only pre-jump values, so these
bounds impose no restriction on its overshoot.

To pass to the energy on the left-hand side, we control the full
post-jump paths by the maximal estimates. Put
\[
 H_T=M_0+2C\bigl(E(u_0)+b_1^2M_0A(T)\bigr).
\]
The maximal energy estimates imply
\begin{equation}\label{balance:uniform-tail-probability}
 \PP\left(\sup_{0\le t\le T}
       \bigl(\|u_n(t)\|_{H^1(\R^3)}\vee\|u(t)\|_{H^1(\R^3)}\bigr)
             \ge R\right)
       \le\frac{2H_T}{R^2},
\end{equation}
uniformly in $n$. On the complementary event,
\eqref{energy:local-lipschitz} bounds the uniform energy difference
by $C(R+R^5)\sup_{0\le t\le T}\|u_n(t)-u(t)\|_{H^1(\R^3)}$.
For fixed $R$ this tends to zero in probability, and the exceptional
probability tends to zero as $R\to\infty$. Therefore
\begin{equation}\label{balance:energy-convergence}
 \sup_{0\le t\le T}|E(u_n(t))-E(u(t))|
                  \longrightarrow0\quad\text{in probability}.
\end{equation}
The same is then true after composition of both paths with their
common stop, since taking that stop cannot enlarge the supremum.

Define the c\`adl\`ag residual
\[
 \mathcal R(t)=E(u(t))-E(u_0)-\mathcal L(t)-\mathcal D(t).
\]
The stopped finite-activity identity \eqref{finite:balance} and
\eqref{balance:martingale-convergence},
\eqref{balance:positive-convergence}, and
\eqref{balance:energy-convergence} imply, for every fixed $K$,
\[
 \sup_{0\le t\le T}|\mathcal R(t\wedge\rho_{K,n})|
                        \longrightarrow0\quad\text{in probability}.
\]
For fixed $K$, this conclusion concerns the residual stopped at the
same time on both sides and requires no convergence of the stopping
times. On the event that both entire paths have $H^1(\R^3)$ norm less
than $K$ up to $T$, the stopped and unstopped residuals agree. By
\eqref{balance:uniform-tail-probability}, for every $\eps>0$,
\[
 \PP\left(\sup_{0\le t\le T}|\mathcal R(t)|>\eps\right)
 \le\frac{2H_T}{K^2}
    +\PP\left(\sup_{0\le t\le T}
          |\mathcal R(t\wedge\rho_{K,n})|>\eps\right).
\]
First let $n\to\infty$ with $K$ fixed, and then let $K\to\infty$.
The residual vanishes identically on $[0,T]$ almost surely. Intersecting
these events over positive integer $T$ proves
\eqref{energy:balance:intro} indistinguishably. The required
integrability of both terms was established in
\eqref{balance:limiting-integrability} and
\eqref{balance:positive-integrability} before taking these limits.

Taking expectations in the resulting identity yields
\begin{equation}\label{energy:expectation}
 \E E(u(t))=E(u_0)+\E\int_0^t\int_Z
                            D_{s,z}(u(s))\nu(\dd z)\dd s.
\end{equation}
The value $u(s)$ can replace $u(s-)$ in this time integral, since a
c\`adl\`ag path has at most countably many discontinuities.
\end{proof}

\section{Scattering with finite total noise intensity}\label{sec:scattering}

Throughout this section, Assumption~\ref{ass:model} holds and
$a\in L^2(0,\infty)$. We prove Theorem~\ref{thm:scattering} for the
global solution supplied by Theorem~\ref{thm:global}. Write
\[
 A(\infty)=Q\int_0^\infty a(s)^2\dd s<\infty.
\]
Mass conservation and \eqref{energy:bound:intro}, followed by monotone
convergence in the terminal time, give
\begin{equation}\label{scat:uniform-energy}
 \E\sup_{t\ge0}\|u(t)\|_{H^1(\R^3)}^2
 \le M(u_0)+2C\bigl(E(u_0)+b_1^2M(u_0)A(\infty)\bigr)
 =:K_*<\infty.
\end{equation}
Thus the full trajectory is bounded in $H^1(\R^3)$ almost surely.
Scattering also requires the forcing restarted at a late time to be
small in the critical spacetime norm. We obtain this pathwise estimate
from convergence of the interaction-picture martingale and a global
bound for its stochastic convolution, then apply deterministic critical
stability.

\begin{lemma}\label{scat:global-convolution}
The process
\begin{equation}\label{scat:martingale}
 \mathcal M(t)=\int_0^t\int_Z
       S(-s)G_{s,z}u(s-)\,\widetilde N(\dd s,\dd z)
\end{equation}
is a square-integrable $H^1(\R^3)$-valued martingale. There is a
limit $\mathcal M_\infty$ in $L^2(\Omega;H^1(\R^3))$ such that
$\mathcal M(t)\to\mathcal M_\infty$ almost surely and in
$L^2(\Omega;H^1(\R^3))$ as $t\to\infty$. Moreover,
\begin{equation}\label{scat:martingale-estimate}
 \E\sup_{t\ge0}\|\mathcal M(t)\|_{H^1(\R^3)}^2
 +\E\|S(\cdot)\mathcal M(\cdot)\|_{Y(0,\infty)}^2
 \le C_VK_*A(\infty).
\end{equation}
The compensator satisfies
\begin{equation}\label{scat:drift-integrable}
 \E\left(\int_0^\infty
       \|B(s)u(s)\|_{H^1(\R^3)}\dd s\right)^2
 \le C_VK_*A(\infty)^2.
\end{equation}
Consequently,
\begin{equation}\label{scat:drift-limit}
 \mathcal D_\infty
 :=\int_0^\infty S(-s)B(s)u(s)\dd s
\end{equation}
is an almost surely absolutely convergent Bochner integral in
$H^1(\R^3)$ and belongs to $L^2(\Omega;H^1(\R^3))$.
\end{lemma}

\begin{proof}
The multiplier bound \eqref{G:bound} and the unitarity of $S(-s)$ give
\begin{align}
 &\E\int_0^\infty\int_Z
       \|S(-s)G_{s,z}u(s-)\|_{H^1(\R^3)}^2\nu(\dd z)\dd s
       \notag\\
 &\qquad\le C_VQ\int_0^\infty a(s)^2
            \E\|u(s-)\|_{H^1(\R^3)}^2\dd s
       \le C_VK_*A(\infty).\label{scat:quadratic-integrability}
\end{align}
The expected total quadratic variation is finite. Hence the terminal
integrals form a Cauchy family in $L^2(\Omega;H^1(\R^3))$ by the
Poisson isometry, defining the limit $\mathcal M_\infty$. The
Hilbert-space martingale convergence theorem gives almost sure
convergence, and the maximal inequality yields the first term in
\eqref{scat:martingale-estimate}. Applied to the increments after $T$,
the same inequality gives
\begin{equation}\label{scat:martingale-tail-mean}
 \E\sup_{t\ge T}
   \|\mathcal M(t)-\mathcal M(T)\|_{H^1(\R^3)}^2
 \le C_VK_*Q\int_T^\infty a(s)^2\dd s.
\end{equation}
The estimate follows first on $[T,R]$ and then on the half-line by
monotone convergence as $R\to\infty$.

For the spacetime term, observe that
\[
 S(t)\mathcal M(t)=\int_0^t\int_Z
      S(t-s)G_{s,z}u(s-)\,\widetilde N(\dd s,\dd z).
\]
Lemma~\ref{lem:poisson}, on any finite interval $[0,R]$, implies
\begin{align*}
 \E\|S(\cdot)\mathcal M(\cdot)\|_{Y(0,R)}^2
 &\le C\E\int_0^R\int_Z
          \|G_{s,z}u(s-)\|_{H^1(\R^3)}^2\nu(\dd z)\dd s\\
 &\le C_VK_*A(\infty).
\end{align*}
The constant does not depend on $R$. Since the tenth power of the
$Y(0,R)$ norm increases to the corresponding integral on $(0,\infty)$,
monotone convergence proves the remaining part of
\eqref{scat:martingale-estimate}.

For the deterministic compensator, \eqref{B:bound} reads
\[
 \|B(s)f\|_{H^1(\R^3)}
 \le C_VQa(s)^2\|f\|_{H^1(\R^3)}.
\]
Thus, almost surely,
\[
 \int_0^\infty\|B(s)u(s)\|_{H^1(\R^3)}\dd s
 \le C_VA(\infty)\sup_{s\ge0}\|u(s)\|_{H^1(\R^3)}.
\]
Squaring and applying \eqref{scat:uniform-energy} proves
\eqref{scat:drift-integrable}. Unitarity of $S(-s)$ gives absolute
convergence and the asserted integrability of
\eqref{scat:drift-limit}.
\end{proof}

The next lemma combines martingale convergence with the global
spacetime estimate. Its conclusion holds for every late starting
time on a fixed path, which will allow the comparison time to depend
on the trajectory.

\begin{lemma}\label{scat:anchored-tail}
Let $m:[0,\infty)\to H^1(\R^3)$ be a c\`adl\`ag function such that
$m(t)\to m_\infty$ in $H^1(\R^3)$ and
$S(\cdot)m(\cdot)\in Y(0,\infty)$. For $t\ge T$, define
\[
 h_T(t)=S(t)\bigl(m(t)-m(T)\bigr).
\]
Then
\begin{equation}\label{scat:anchored-tail-limit}
 \lim_{T\to\infty}\|h_T\|_{\mathcal X([T,\infty))}=0.
\end{equation}
\end{lemma}

\begin{proof}
For the supremum term, unitarity gives
\begin{align*}
 \sup_{t\ge T}\|h_T(t)\|_{H^1(\R^3)}
 &=\sup_{t\ge T}\|m(t)-m(T)\|_{H^1(\R^3)}\\
 &\le2\sup_{t\ge T}\|m(t)-m_\infty\|_{H^1(\R^3)}\longrightarrow0.
\end{align*}
For the spacetime term, write
\[
 h_T(t)=S(t)m(t)-S(t)m_\infty
              +S(t)\bigl(m_\infty-m(T)\bigr).
\]
The homogeneous Strichartz estimate, with a constant independent of $T$,
therefore yields
\begin{align}
 \|h_T\|_{Y(T,\infty)}
 &\le\|S(\cdot)m(\cdot)\|_{Y(T,\infty)}
       +\|S(\cdot)m_\infty\|_{Y(T,\infty)}\notag\\
 &\quad+C\|m_\infty-m(T)\|_{H^1(\R^3)}.\label{scat:anchored-tail-bound}
\end{align}
The first two terms tend to zero by absolute continuity of the
spacetime integrals: the first is finite by assumption, and the second
by homogeneous Strichartz for the fixed vector $m_\infty$. The last
term tends to zero by convergence of $m(T)$. This proves
\eqref{scat:anchored-tail-limit}.
\end{proof}

Define, for $t\ge T$,
\begin{equation}\label{scat:restarted-noise}
 \begin{split}
 \Xi_T(t)={}&S(t)\bigl(\mathcal M(t)-\mathcal M(T)\bigr)
            +\int_T^t S(t-s)B(s)u(s)\dd s\\
 ={}&\int_T^t\int_Z S(t-s)G_{s,z}u(s-)
                 \,\widetilde N(\dd s,\dd z)
            +\int_T^t S(t-s)B(s)u(s)\dd s.
 \end{split}
\end{equation}
The first line defines the restarted path simultaneously for all
$T$; the second is its stochastic-integral representation for
deterministic $T$. A time selected after fixing the path is always
interpreted through the first line. The initial value $u(T)$ includes
the jump at $T$, whereas the restarted integrals exclude that jump.
Equivalently, if $\Psi(t)=S(t)\mathcal M(t)$, the martingale contribution
to $\Xi_T$ is $\Psi(t)-S(t-T)\Psi(T)$.

By Lemmas~\ref{scat:global-convolution} and \ref{scat:anchored-tail},
the martingale contribution tends to zero in
$\mathcal X([T,\infty))$ almost surely. For the compensator,
Minkowski's inequality and homogeneous Strichartz give
\begin{equation}\label{scat:drift-tail}
 \left\|\int_T^t S(t-s)B(s)u(s)\dd s
          \right\|_{\mathcal X([T,\infty))}
 \le C\int_T^\infty\|B(s)u(s)\|_{H^1(\R^3)}\dd s
 \longrightarrow0
\end{equation}
almost surely. For the spacetime estimate, apply homogeneous
Strichartz at each $s\ge T$ to
$\mathbf1_{[s,\infty)}(t)S(t-s)B(s)u(s)$, then integrate in $s$.
The supremum estimate follows from unitarity. Consequently,
\begin{equation}\label{scat:noise-tail}
 \lim_{T\to\infty}\|\Xi_T\|_{\mathcal X([T,\infty))}=0
 \qquad\text{almost surely}.
\end{equation}
The limit holds on one event of full probability for all starting
times, since the two lemmas apply to complete paths.

For deterministic starting times, the same estimates also give a
quantitative mean-square bound. Define
\[
 A_T=Q\int_T^\infty a(s)^2\dd s.
\]
The Poisson convolution on $[T,R]$ starts with value zero at $T$.
Using Lemma~\ref{lem:poisson} and then increasing $R$ gives
\[
 \begin{split}
 &\E\left\|\int_T^t\int_ZS(t-s)G_{s,z}u(s-)
          \widetilde N(\dd s,\dd z)\right\|_{\mathcal X([T,\infty))}^2\\
 &\qquad\le C_VQ\int_T^\infty a(s)^2
       \E\|u(s-)\|_{H^1(\R^3)}^2\dd s
       \le C_VK_*A_T.
 \end{split}
\]
For the compensator, \eqref{str:lone} gives the pathwise bound
\[
 \left\|\int_T^tS(t-s)B(s)u(s)\dd s
                             \right\|_{\mathcal X([T,\infty))}
 \le C_VA_T\sup_{s\ge T}\|u(s)\|_{H^1(\R^3)}.
\]
Squaring, taking expectations, and combining the two contributions gives
\begin{equation}\label{scat:noise-tail-quantitative}
 \E\|\Xi_T\|_{\mathcal X([T,\infty))}^2
                 \le C_VK_*(A_T+A_T^2).
\end{equation}
Thus the forcing is small in mean square at deterministic late times.
The scattering proof uses the pathwise limit
\eqref{scat:noise-tail}, since the stability threshold depends on the
$H^1(\R^3)$ bound of the entire trajectory.

\begin{proof}[Proof of Theorem~\ref{thm:scattering}]
Work on the full-probability event where the solution is global,
conserves mass, has the bounded trajectory supplied by
\eqref{scat:uniform-energy}, and satisfies \eqref{scat:noise-tail}.
Fix a sample point in this event and set
\[
 K=1+\sup_{t\ge0}\|u(t)\|_{H^1(\R^3)}<\infty.
\]
Lemma~\ref{lem:detglobal} bounds the global $\mathcal X$ norm of
every solution of \eqref{det:eq} starting from data of
$H^1(\R^3)$ norm at most $K$. Together with
Lemma~\ref{lem:stability}, this gives a positive perturbation threshold
$\delta_*(K)$ and comparison bounds depending only on $K$, uniformly
in the starting and terminal times.

By \eqref{scat:noise-tail}, choose a finite $T$ for which
\[
 \|\Xi_T\|_{\mathcal X([T,\infty))}<\delta_*(K).
\]
The time $T$ is used only in the pathwise deterministic comparison;
the first line of \eqref{scat:restarted-noise} defines the external
term for this choice. Let $v_T$ solve
\eqref{det:eq} with $v_T(T)=u(T)$. For each finite $R>T$, the
restarted mild identity is
\begin{equation}\label{scat:comparison-equation}
 u(t)=S(t-T)u(T)-i\int_T^t S(t-s)F(u(s))\dd s+\Xi_T(t),
 \qquad T\le t\le R.
\end{equation}
The initial data for $u$ and $v_T$ agree, and restricting $\Xi_T$ to
$[T,R]$ cannot increase its norm. Applying deterministic critical
stability on this finite interval gives
\begin{equation}\label{scat:finite-terminal-bound}
 \|u\|_{\mathcal X([T,R])}\le C(K),\qquad R>T,
\end{equation}
where the right-hand side is independent of $R$.

The c\`adl\`ag regularity of $\Xi_T$ is sufficient for stability:
subtracting it from \eqref{scat:comparison-equation} leaves the
continuous Duhamel expression, and the comparison estimates use only
the $\mathcal X$ norm of the external term. In the notation of
Lemma~\ref{lem:stability}, the number of comparison
intervals is bounded by $J_K$, and the endpoint errors obey
$e_{j+1}+h\le(C_1+1)(e_j+h)$ with
$h=\|\Xi_T\|_{\mathcal X([T,\infty))}$ and $e_0=0$.
The choice of $\delta_*(K)$ ensures the smallness hypothesis at every
step, including a last interval truncated at $R$. Hence
\eqref{scat:finite-terminal-bound} is uniform in $R$.

Letting $R\to\infty$ in \eqref{scat:finite-terminal-bound} gives
$u\in\mathcal X([T,\infty))$. Theorem~\ref{thm:global} supplies the
finite $Y$ norm before $T$. We have proved
\begin{equation}\label{scat:whole-critical}
 \|u\|_{Y(0,\infty)}<\infty
 \qquad\text{almost surely}.
\end{equation}

We next construct the scattering state. By Sobolev embedding and the
quintic estimate from Section~\ref{sec:analytic}, for any interval $I$,
\begin{equation}\label{scat:nonlinearity-tail}
 \|F(u)\|_{\mathcal N(I)}
 \le C\|u\|_{L^{10}(I\times\R^3)}^4\|u\|_{Y(I)}
 \le C\|u\|_{Y(I)}^5.
\end{equation}
For the differentiated term, use
$|\nabla F(u)|\le5|u|^4|\nabla u|$ and apply H\"older's inequality
with the four undifferentiated factors in $L^{10}(I\times\R^3)$
and the gradient in $L^{10}(I;L^{30/13}(\R^3))$.
The undifferentiated term is estimated with $u$ in place of
$\nabla u$ in the last factor.
It follows from \eqref{scat:whole-critical} that
$\|F(u)\|_{\mathcal N(T,\infty)}\to0$. Dual Strichartz gives
\begin{equation}\label{scat:nonlinear-cauchy}
 \left\|\int_r^s S(-\tau)F(u(\tau))\dd\tau
                \right\|_{H^1(\R^3)}
 \le C\|F(u)\|_{\mathcal N(r,s)},\qquad 0\le r<s.
\end{equation}
The dual Strichartz construction therefore gives convergence of
this integral in $H^1(\R^3)$ as its upper endpoint tends to infinity.
This uses the vanishing $\mathcal N$ norm of the tails and requires
no absolute $H^1(\R^3)$ integrability of $F(u)$.

Define
\begin{equation}\label{scat:state}
 u_+=u_0-i\int_0^\infty S(-s)F(u(s))\dd s
                     +\mathcal M_\infty+\mathcal D_\infty.
\end{equation}
Multiplying \eqref{main:mild} by $S(-t)$, and using
\eqref{scat:nonlinear-cauchy} and Lemma~\ref{scat:global-convolution},
gives the explicit difference
\[
 \begin{split}
 S(-t)u(t)-u_+
   ={}&i\int_t^\infty S(-s)F(u(s))\dd s
                  +\mathcal M(t)-\mathcal M_\infty\\
     &-\int_t^\infty S(-s)B(s)u(s)\dd s.
 \end{split}
\]
Consequently,
\begin{equation}\label{scat:three-tails}
 \begin{split}
 &\sup_{t\ge T}\|u(t)-S(t)u_+\|_{H^1(\R^3)}\\
 &\quad\le C\|u\|_{Y(T,\infty)}^5
      +\sup_{t\ge T}\|\mathcal M(t)-\mathcal M_\infty\|_{H^1(\R^3)}
      +\int_T^\infty\|B(s)u(s)\|_{H^1(\R^3)}\dd s.
 \end{split}
\end{equation}
The three tails vanish by the finite critical norm, convergence of
the $H^1(\R^3)$ martingale, and absolute integrability of the
compensator, respectively. Thus scattering holds uniformly over all
times after $T$ as $T\to\infty$.
The state is measurable, for it is the almost sure $H^1(\R^3)$ limit of
$S(-n)u(n)$, $n\in\mathbb N$; it can be set equal to zero on the
exceptional null set. It is unique by unitarity of $S(t)$. Moreover,
\begin{equation}\label{scat:state-norm}
 \|u_+\|_{H^1(\R^3)}
 =\lim_{t\to\infty}\|u(t)\|_{H^1(\R^3)}
 \le\sup_{t\ge0}\|u(t)\|_{H^1(\R^3)}.
\end{equation}
Hence $u_+\in L^2(\Omega;H^1(\R^3))$ by
\eqref{scat:uniform-energy}. Passing to the limit in the conserved
$L^2$ norm also gives $M(u_+)=M(u_0)$ almost surely.

To obtain the mean-square assertion, observe that almost sure
scattering implies
\[
 \sup_{t\ge T}\|u(t)-S(t)u_+\|_{H^1(\R^3)}^2\longrightarrow0
 \quad\text{almost surely}.
\]
By \eqref{scat:state-norm}, this random variable is bounded by
\begin{equation}\label{scat:mean-domination}
 4\sup_{t\ge0}\|u(t)\|_{H^1(\R^3)}^2
 \le4M(u_0)+8\sup_{t\ge0}E(u(t)).
\end{equation}
This bound is integrable. The supremum is measurable since the
error has c\`adl\`ag $H^1(\R^3)$ paths and can be computed
over a countable dense set of times. Dominated convergence proves
\eqref{scatter:mean}.

It remains to identify the expected kinetic energy at infinity. For
every $f\in H^1(\R^3)$,
\begin{equation}\label{scat:free-six-decay}
 \lim_{t\to\infty}\|S(t)f\|_{L^6(\R^3)}=0.
\end{equation}
Indeed, choose $f_n\in C_c^\infty(\R^3)$ converging to $f$ in
$H^1(\R^3)$. Sobolev embedding and unitarity imply
\[
 \|S(t)(f-f_n)\|_{L^6(\R^3)}
 \le C\|\nabla(f-f_n)\|_{L^2(\R^3)},
\]
uniformly in $t$, while the dispersive estimate yields
\[
 \|S(t)f_n\|_{L^6(\R^3)}
 \le C|t|^{-1}\|f_n\|_{L^{6/5}(\R^3)}\longrightarrow0.
\]
First let $t\to\infty$ and then $n\to\infty$. Applying
\eqref{scat:free-six-decay} to $u_+$ pathwise, and using
$H^1(\R^3)$ scattering, gives
\[
 \|u(t)\|_{L^6(\R^3)}\longrightarrow0,
 \qquad
 \|\nabla u(t)\|_{L^2(\R^3)}
       \longrightarrow\|\nabla u_+\|_{L^2(\R^3)}
 \quad\text{almost surely}.
\]
Consequently,
\begin{equation}\label{scat:energy-limit}
 E(u(t))\longrightarrow\frac12\|\nabla u_+\|_{L^2(\R^3)}^2
 \quad\text{almost surely and in }L^1(\Omega).
\end{equation}
The $L^1(\Omega)$ convergence follows by domination with the
integrable random variable $\sup_{s\ge0}E(u(s))$. Thus this passage
also uses only the established energy bound, without requiring a
sixth moment of $\|u_+\|_{H^1(\R^3)}$.

Taking expectations in Proposition~\ref{prop:balance} gives, for
each finite $t$,
\begin{equation}\label{scat:finite-mean-energy}
 \E E(u(t))=E(u_0)+\frac12\E\int_0^t a(s)^2
       \int_Z\|u(s)\nabla A_z\|_{L^2(\R^3)}^2\nu(\dd z)\dd s.
\end{equation}
Here the martingale has zero expectation, and the nonnegative jump
integral is evaluated by compensation. Replacing $u(s-)$ by $u(s)$
under $\dd s$ leaves the integral unchanged. Its total expectation
is bounded by
\begin{align}
 &\E\int_0^\infty a(s)^2
       \int_Z\|u(s)\nabla A_z\|_{L^2(\R^3)}^2\nu(\dd z)\dd s
       \notag\\
 &\qquad\le b_1^2M(u_0)Q\int_0^\infty a(s)^2\dd s
       =b_1^2M(u_0)A(\infty).\label{scat:energy-injection-bound}
\end{align}
Finally, apply the $L^1(\Omega)$ convergence in
\eqref{scat:energy-limit} to the left-hand side of
\eqref{scat:finite-mean-energy}, and monotone convergence to its
nonnegative right-hand integral. This proves \eqref{scatter:energy}.
\end{proof}

\section*{Funding}
This work was partially supported by the National Key Research and
Development Program of China (Grant No.~2023YFC2206100), the Fundamental
Research Funds for the Central Universities (Grant No.~2026BRSXB003),
the National Natural Science Foundation of China (Grant No.~12231008),
and the Postdoctoral Fellowship Program of China Postdoctoral Science
Foundation (Grant No.~GZC20261670).

\section*{Declaration of competing interest}
The authors declared that they have no conflicts of interest to this work.

\section*{Declaration of artificial intelligence use}
During the preparation of this manuscript, the authors used ChatGPT
as an auxiliary tool for language editing, organization of arguments,
proof auditing, and manuscript preparation, including preliminary checks
of mathematical derivations. All mathematical statements, proofs,
references, and conclusions appearing in the manuscript were independently
verified by the authors, who take full responsibility for the content.

\end{document}